\documentclass [12pt,oneside,a4paper,mathscr]{amsart}
\usepackage{float}
\usepackage[margin=2.5cm]{geometry}
\usepackage{amssymb,mathtools,amsmath,extarrows}
\usepackage[pagewise]{lineno}
\numberwithin{equation}{section}
\usepackage{cite}
\usepackage{url}
\usepackage{graphicx}
\usepackage[all]{xy}

\usepackage{tikz}
\usetikzlibrary{backgrounds}
\usetikzlibrary{decorations}
\usetikzlibrary{decorations.pathreplacing} 
\usetikzlibrary{arrows,cd}
\usetikzlibrary{shapes,shapes.geometric,shapes.misc}

\tikzstyle{tikzfig}=[baseline=-0.25em,scale=0.5]

\pgfkeys{/tikz/tikzit fill/.initial=0}
\pgfkeys{/tikz/tikzit draw/.initial=0}
\pgfkeys{/tikz/tikzit shape/.initial=0}
\pgfkeys{/tikz/tikzit category/.initial=0}

\pgfdeclarelayer{edgelayer}
\pgfdeclarelayer{nodelayer}
\pgfsetlayers{background,edgelayer,nodelayer,main}

\tikzstyle{none}=[inner sep=0mm]

\newcommand{\tikzfig}[1]{%
{\tikzstyle{every picture}=[tikzfig]
\IfFileExists{#1.tikz}
  {\input{#1.tikz}}
  {%
    \IfFileExists{./figures/#1.tikz}
      {\input{./figures/#1.tikz}}
      {\tikz[baseline=-0.5em]{\node[draw=red,font=\color{red},fill=red!10!white] {\textit{#1}};}}%
  }}%
}

\tikzstyle{every loop}=[]

\tikzstyle{edge_arrow}=[<-]
\tikzstyle{edge_dot}=[--, fill=none]
\tikzstyle{edge_arrow_2}=[->]

\theoremstyle{definition}
\newtheorem{thm}{Theorem}[section]
\newtheorem{cor}[thm]{Corollary}
\newtheorem{lem}[thm]{Lemma}
\newtheorem{exa}[thm]{Example}
\newtheorem{prop}[thm]{Proposition}
\newtheorem{defi}[thm]{Definition}
\newtheorem{rem}[thm]{Remark}

\DeclareMathOperator{\Coh}{\mathrm{Coh}}

\DeclareMathOperator{\GL}{\mathrm{GL}}
\DeclareMathOperator{\Hom}{\mathrm{Hom}}
\DeclareMathOperator{\Spec}{\mathrm{Spec}}
\DeclareMathOperator{\End}{\mathrm{End}}
\DeclareMathOperator{\Ext}{\mathrm{Ext}}

\DeclareMathOperator{\HO}{\mathrm{H}}

\DeclareMathOperator{\supp}{\mathrm{supp}}

\DeclareMathOperator{\Sim}{\mathrm{Sim}}

\DeclareMathOperator{\Aut}{\mathrm{Aut}}
\DeclareMathOperator{\Cone}{\mathrm{Cone}}
\DeclareMathOperator{\nil}{\mathrm{nil}}

\DeclareMathOperator{\Stab}{\mathrm{Stab}}
\DeclareMathOperator{\Rep}{\mathrm{Rep}}
\DeclareMathOperator{\rep}{\mathrm{rep}}
\DeclareMathOperator{\thick}{\mathrm{thick}}

\DeclareMathOperator{\Tw}{Tw}
\DeclareMathOperator{\reg}{reg}
\DeclareMathOperator{\Tot}{Tot}

\newcommand{\wt}[1]{\widetilde{#1}}

\newcommand{\mr}[1]{\mathrm{#1}}
\newcommand{\mb}[1]{\mathbb{#1}}

\newcommand{\mc}[1]{\mathcal{#1}}

\newcommand{\UnA}[1]{\mc{U}^n(#1 \ca)^{\Phi}}

\newcommand{\ho}{\mathrm{H}}
\newcommand{\pp}{\mb{P}^1\times\mb{P}^1}
\newcommand{\ca}{\mc{A}}
\newcommand{\CD}{\mc{D}}

\newcommand{\y}{\gamma}

\newcommand{\BC}{\mathbb{C}}
\newcommand{\BE}{\mathbb{E}}

\newcommand{\CP}{\mathbb{P}}

\newcommand{\BZ}{\mathbb{Z}}
\newcommand{\CC}{\mathcal{C}}

\newcommand{\co}{\mathcal{O}}

\newcommand{\wbar}{\overline}

\newcommand{\wtil}{\widetilde}

\newcommand{\RHom}{\mathbb{R}\mathrm{Hom}}
\newcommand{\FFF}{\mb{P}^1\times \mb{P}^1}
\newcommand{\ff}{\mathbb{F}_0}

\begin{document}

\title{Invariant stability conditions of Local $\mb{P}^1\times \mb{P}^1$ (after Del Monte-Longhi)}
\author{Yirui Xiong}
\address{School of Sciences, Southwest Petroleum University, 610500 Chengdu, People’s Republic of China}
\email{yiruimee@icloud.com}

\maketitle

\begin{abstract}
    Let $X$ be the total space of the canonical bundle of $\pp$, we study an invariant subspace of stability conditions on $X$  under an autoequivalence of $D^b(X)$. We describe the complete set of stable objects with respect to the invariant stability conditions and characterize the space of invariant stability conditions.
\end{abstract}
\section{Introduction}
\subsection{Background} Inspired by the Douglas' work on $\Pi$-stability for D-branes, Bridgeland introduced the notion of the stability condition on a triangulated category in \cite{Bri07}. It was shown in \cite{Bri07} that to any triangulated category $\CD$, one can associate a complex manifold $\Stab(\CD)$ which parameterises stability conditions on $\CD$. Recall a stability condition on $\CD$ is a pair  $\sigma = (Z, \ca)$, where $\ca$  is a full subcategory of $\CD$ called the heart, and $Z$ is a group homomorphism called the central charge from the Grothendieck group $K_0(\ca)$ to $\mb{C}$  which satisfies the Harder-Narasimhan property\cite{Bri07} (see Definition \ref{defi:HN}). Bridgeland showed that  if $\Stab(\CD)$ is nonempty, then the forgetfull map $\Stab(\CD) \rightarrow \Hom(K_0(\CD),\mb{C})$ which sends $(Z,\mc{P})$ to $Z$ is a local homeomorphism. Given the triangulated category $\CD$, one can ask the following three questions:
\begin{enumerate}
{\it
    \item Can we find a stability condition on $\CD$?
    \item What is $\Stab(\CD)$ as a complex manifold?
    \item Given a stability condition $\sigma$, can we count the set of (semi)stable objects in $\CD$ for $\sigma$?
}
\end{enumerate}
 So far much progress towards the first two questions has been made for the derived categories of projective and quasi-projective (local) varieties \cite{Bri08K3,Bri06_3fold,Bri09,Li19, HW19,AW22}. 

The answer to the final question is usually very hard for both projective and quasi-projective varieties. When $\CD$ is a Calabi-Yau category of dimension $3$, it is  related to the Donaldson-Thomas invariants \cite{JoSong12, KonSoi08}. 

We study the space of stability conditions for local $\mb{P}^1 \times \mb{P}^1$ in this paper, i.e., the total space $X$ of canonical bundle over $\mb{P}^1 \times \mb{P}^1$. The work is a mathematical interpretation of the work of Del Monte-Longhi \cite{DL22}. In their paper, the physicists found that there were  surprisingly complete answers to all questions in the above when we restrict to an invariant subspace of stability conditions under an autoequivalence of $D^b(X)$. 

\subsection{Results} Denote by $\pi:X= \Tot \omega_{\pp} \rightarrow \FFF$ the bundle projection map and $p_i: \FFF \rightarrow \mb{P}^1$, $i = 1,\ 2$ the projection maps to each component. Write
$\co(a,b)$ for the line bundle  $p_1^*\co(a) \otimes p_2^* \co(b)$. There is a full and strong exceptional sequence of line bundles on $\mb{P}^1\times \mb{P}^1$:
\[
    \mb{E} =(E_0, E_1,E_2,E_3):= \left(\co(0,0), \co(1,0),  \co(1,1), \co(2,1)\right),
\]
which generates $D^b(\FFF)$, and
\[
    T = \bigoplus T_i := \bigoplus_i \pi^*E_i
\]
is a tilting object in $D^b(X)$. Then by the derived Morita theory \cite{Keller94} $X$ is derived equivalent to a non-commutative algebra $A= \End(T)$ via the  functor $\RHom_X(T,-)$. Normally we present $A$ as the path algebra of a quiver with relations, then $A = \oplus_{i\geq 0}A_i$ has a natural grading by the length of paths. Denote by $D^b_0(X)$ the full subcategory of $D^b(X)$ consisting of objects supported on $\FFF$,  and $D^b_0(A)$ the full subcategory of $D^b(A)$ consisting of objects whose cohomology modules are nilpotent, here we say a right $A$-module $M$ is nilpotent if there exists $n>0$ such that $MA_n = 0$. Then $\RHom_X(T,-)$ restricts to an equivalence between $\CC = D^b_0(X)$ and $D^b_0(A)$.

The presentation of the algebra $A$ as the path algebra of a quiver $Q$ subject to relations is as follows,  the nodes of $Q$ correspond to the line bundles $T_i$, and the number of arrows between two nodes can either be calculated from the irreducible maps from $T_i$ to $T_j$, or the  first extension group of the pair of simple modules associated with the vertices. The quiver of $A$ will be
\begin{equation}\label{qui:local_f0}
    \xymatrixcolsep{4pc}
    \xymatrix{
        0 \ar@<0.5ex>[r] \ar@<-0.5ex>[r] & 1 \ar@<0.5ex>[d] \ar@<-0.5ex>[d] \\
        3 \ar@<0.5ex>[u] \ar@<-0.5ex>[u] & 2 \ar@<0.5ex>[l] \ar@<-0.5ex>[l]
    }
\end{equation}
The symmetry of the shape of the quiver suggests that there should be an autoequivalence of $D^b(X)$, denoted by $\Psi$, which cyclically permutes  the simple modules associated with the vertices.  We will realize $\Psi$ explicitly in Section \ref{sec:auto} and let $\Phi = \Psi^2$. $\Phi$ restricts to be an autoequivalence of $D^b_0(X)$. The space of stability conditions on $D^b_0(X)$ is denoted by $\Stab(X)$. Following Del Monte-Longhi \cite{DL22} we consider the space of stability conditions $\Stab(X)^{\Phi}$ which are invariant under $\Phi$, which is called collimination chamber in their paper.

Let $\varphi$ be the automorphism of Grothendieck group $K_0(X)$ induced by $\Phi$. Then $\Stab(X)^{\Phi}$ is locally modelled on $\Hom(K_0(X), \mb{C})^{\varphi}$, the invariant central charges under $\varphi$. Note that $K_0(X) \cong \mb{Z}^4$ has  a basis $\gamma_i = [\pi^* E_i]$, $i=0,\cdots,3$ which corresponds to the vertices of $Q$. Therefore $\Hom(K_0(X),\mb{C})^{\varphi}\cong \mb{C}^2$. 

The construction of a stability condition for $D^b_0(X)$ is simple: denote by $\ca$ the heart in $D^b_0(X)$ by pulling back the standard t-structure on $D^b_0(A)$, i.e. $\ca$ is equivalent to $\text{mod}_0\text{-}A$ via the functor $\RHom_X(T,-)$ where $\text{mod}_0\text{-}A$ is the category of nilpotent modules over $A$. We have $K_0(\ca)\cong K_0(X)$ by sending $S_i$ to $\gamma_i$. Since $\ca$ is of finite-length and has a finite set of simple objects $S_i$, let $Z:K_0(\ca)\rightarrow \mb{C}$ assign the class of each simple object $[S_i]$ to the semi-closed upper half plane
\[
    H = \{z = r\mr{exp}(i\pi\phi) \vert r>0, 0 <\phi\leq 1\} \subset \mb{C},
\] 
then $(Z,\ca)$ satisfies the Harder-Narasimhan property automatically and is therefore a stability condition. Moreover  we write $\mc{U}(\ca)^{\Phi}$ for the subset of $\Stab(X)^{\Phi}$ consisting of $\Phi$-invariant stability conditions with heart $\ca$, since such stability condition is uniquely determined by assigning each simple object to $H$, then $\mc{U}(\ca)^{\Phi}$ is isomorphic to $H^2$. Such stability conditions are called algebraic in  \cite{Bri06_3fold}\cite{BM11}. 

First we characterize the stable objects for $\sigma\in\mc{U}(\ca)^{\Phi}$: note that the Kronecker quiver $K_2$
    \[
        \xymatrix{
        0 \ar@<0.5ex>[r] \ar@<-0.5ex>[r] & 1
        }
    \]
    can be embedded into $Q$ (\ref{qui:local_f0}) in 4 different ways. Therefore $\rep(K_2)$ embeds into  $\ca$ as full subcategories. For a stability condition $\sigma\in\mc{U}(\ca)^{\Phi}$, $\sigma$  reduces to be a stability function (Definition \ref{def:stab_func}) $\bar{\sigma}$ on $\rep(K_2)$. We are able to show that
    \begin{lem}[=Lemma \ref{lem:restr_stab}]
        The stable objects in $\rep(K_2)$ with respect to $\bar{\sigma}$ are stable in $\ca$ with respect to $\sigma$.
    \end{lem}
    The stable objects in $\rep(K_2)$ are well known to be the indecomposable representations (with respect to certain stability functions), and their images in $\ca$ are called the objects of special Kronecker types I and II (Definition \ref{def:spe_kron}). 
    The main result is that  these are in fact all the stable objects for $\sigma\in\mc{U}(\ca)^{\Phi}$:
\begin{thm}[=Theorem \ref{thm:stab_obj}]
   Take $\sigma =(Z,\ca)\in \mc{U}(\ca)^{\Phi}$. Then  the stable objects in $D^b_0(X)$  for $\sigma$ and their classes in Grothendieck group (up to a sign) are as follows:
    \begin{enumerate}
        \item if $\arg Z(\gamma_0) < \arg Z(\gamma_1)$, then the classes of stable objects  are 
            \begin{eqnarray*}
                n \gamma_0+(n+1)\gamma_1,& (n+1)\gamma_0+n\gamma_1,\\
                n \gamma_2+(n+1)\gamma_3,& (n+1)\gamma_2+n\gamma_3,\\
                \gamma_0+\gamma_1,& \gamma_2+\gamma_3.
            \end{eqnarray*}
            each of the first 4  classes corresponds to a unique stable object (up to a shift of degree), and each of the last two classes corresponds to a $\mb{P}^1$-family of stable objects;
        \item if $\arg Z(\gamma_1) < \arg Z(\gamma_0)$, then the classes of stable objects  are
        \begin{eqnarray*}
                n \gamma_1+(n+1)\gamma_2,& (n+1)\gamma_1+n\gamma_2,\\
                n \gamma_3+(n+1)\gamma_0,& (n+1)\gamma_3+n\gamma_0,\\
                \gamma_1+\gamma_2,& \gamma_3+\gamma_0.
            \end{eqnarray*}
            each of the first 4 classes corresponds to a unique stable object(up to a shift of degree), and each of the last two classes corresponds to a $\mb{P}^1$-family of stable objects;
        \item if $\arg Z(\gamma_0) = \arg Z(\gamma_1)$, then the classes of stable objects are
        \[
            \gamma_0, \ \gamma_1,\ \gamma_2,\ \gamma_3.
        \]
        each class corresponds to a unique stable object $S_i$.
    \end{enumerate} 
\end{thm}
This answers the question of counting the stable objects for the local $\FFF$. We proceed to characterize a connected component of $\Stab(X)^{\Phi}$.

Let $K_0(X)^{-\varphi}$ be the subgroup of $K_0(X)$ whose elements are antisymmetric under $\varphi$. We denote by $\overline{K_0(X)} = K_0(X)/K_0(X)^{-\varphi}$ the quotient group. Note that there is an isomorphism 
\[
    \Hom_{\mb{Z}}(\overline{K_0(X)},\mb{C}) \longrightarrow \Hom_{\mb{Z}}(K_0(X),\mb{C})^{\varphi}.
\]
We write $\Delta\subset \overline{K_0(X)}$  for the image of the set of the  classes of the stable objects for $\sigma\in \mc{U}(\ca)^{\Phi}$ in the quotient group. The connected component of $\Stab(X)^{\Phi}$ which contains $\mc{U}(\ca)^{\Phi}$ is denoted by $\bigl(\Stab(X)^{\Phi}\bigr)_0$, we have

 \begin{figure}[h]
\begin{center}
\begin{tikzpicture}[scale = 0.5]
	\begin{pgfonlayer}{nodelayer}
		\node [style=none] (0) at (-6, 0) {};
		\node [style=none] (1) at (6, 0) {};
		\node [style=none] (2) at (0, 6) {};
		\node [style=none] (4) at (0, -6) {};
		\node [style=none] (5) at (-6, 6) {};
		\node [style=none] (6) at (6, -6) {};
		\node [style=none] (7) at (-6, 5) {};
		\node [style=none] (8) at (-6, 4.5) {};
		\node [style=none] (9) at (-6, 4) {};
		\node [style=none] (10) at (-6, 3) {};
		\node [style=none] (11) at (-5, 6) {};
		\node [style=none] (12) at (-4.5, 6) {};
		\node [style=none] (13) at (-4, 6) {};
		\node [style=none] (14) at (-3, 6) {};
		\node [style=none] (15) at (-2, 6) {};
		\node [style=none] (16) at (-6, 2) {};
		\node [style=none] (17) at (6, -5) {};
		\node [style=none] (18) at (6, -4.5) {};
		\node [style=none] (19) at (6, -4) {};
		\node [style=none] (20) at (5, -6) {};
		\node [style=none] (21) at (4.5, -6) {};
		\node [style=none] (22) at (4, -6) {};
		\node [style=none] (23) at (6, -3) {};
		\node [style=none] (24) at (6, -2) {};
		\node [style=none] (25) at (3, -6) {};
		\node [style=none] (26) at (2, -6) {};
		\node [style=none] (27) at (5.5, -6) {$\cdots$};
		\node [style=none] (28) at (6, -5.5) {$\vdots$};
		\node [style=none] (29) at (-6, 5.5) {$\vdots$};
		\node [style=none] (30) at (-5.5, 6) {$\cdots$};
		\node [style=none] (31) at (6, 0.5) {$Z(\gamma_1) = Z(\gamma_3)$};
		\node [style=none] (32) at (0.75, 6.5) {$Z(\gamma_0)=Z(\gamma_2)$};
	\end{pgfonlayer}
	\begin{pgfonlayer}{edgelayer}
		\draw (2.center) to (4.center);
		\draw (0.center) to (1.center);
		\draw (15.center) to (26.center);
		\draw (14.center) to (25.center);
		\draw (13.center) to (22.center);
		\draw (12.center) to (21.center);
		\draw (11.center) to (20.center);
		\draw (5.center) to (6.center);
		\draw (7.center) to (17.center);
		\draw (8.center) to (18.center);
		\draw (9.center) to (19.center);
		\draw (10.center) to (23.center);
		\draw (16.center) to (24.center);
	\end{pgfonlayer}
\end{tikzpicture}
\caption{Real slice of $\mc{H}^{\reg}$}
    \label{fig:hreg}
\end{center}
\end{figure}

\begin{thm}[=Theorem \ref{thm:local_homeo} and \ref{thm:stab_cover}] \label{thm:main_thm2}
The image of the forgetful map
\[
    \mc{Z}: \bigl(\Stab(X)^{\Phi}\bigr)_0 \rightarrow \Hom(\overline{K_0(X)},\mb{C})
\]
factors through
\begin{equation}\label{map:cover}
\mc{Z}: \bigl(\Stab(X)^{\Phi}\bigr)_0 \rightarrow \mc{H}^{\reg} 
\end{equation}
where
\[
\mc{H}^{\reg} := \Hom(\overline{K_0(X)}, \mb{C}) \setminus \bigcup_{\pmb{v}\in \Delta} \pmb{v}^{\perp}, 
\]
is the hyperplane complement of  $\pmb{v}^{\perp} := \{ Z\in  \Hom(\overline{K_0(X)},\mb{C}) \mid Z(\pmb{v}) = 0\}$ for $\pmb{v}\in\Delta$. Moreover in $(\ref{map:cover})$ $\mc{Z}$ is a covering map.
\end{thm}
\subsection{Relation with \cite{DL22} and related works}
Finally we explain the relation between our work with \cite{DL22}. We keep the notations as above and introduce the notations for normalized stability conditions  $\mc{U}^n(\ca)^{\Phi}\subset \mc{U}(\ca)^{\Phi}$ by
\[
    \mc{U}^n(\ca)^{\Phi} = \{(Z,\mc{P}): Z(\delta) = i\},
\]
where $\delta$ is the class of a skyscrapper sheaf $\co_x$, $x\in \FFF$. Note that $Z(\gamma_0) + Z(\gamma_2) = Z(\gamma_1) + Z(\gamma_3)$ lies on the imaginary axis. It was shown by Closset-Del Zotto \cite[Appendix D]{CZ19} that there is a unique stable object in each slicing $\mc{P}(\phi)$ where $\phi \neq \frac{1}{2}+n$ ($n\in \BZ$), and each such object corresponds to a representation of the Kronecker quiver. 

In this paper we analyze the special slicing $\mc{P}(\frac{1}{2})$ in detail. For simplicity, we restrict to the subset $\mc{U}^n(\ca)^{\Phi}_+ := \bigl\{ (Z,\mc{P}): \arg Z(\gamma_0) < \arg Z(\gamma_1)\bigr\}$. We show that  each stable object in $\mc{P}(\frac{1}{2})$ is isomorphic to $s_*\co_{\{y\} \times \mb{P}^1}$ or $s_*\co_{\{y\}\times \mb{P}^1}(-1)[1]$ for $y\in \mb{P}^1$ (Theorem \ref{stabobj}). The case for $\mc{U}^n(\ca)^{\Phi}_- := \bigl\{ (Z,\mc{P}): \arg Z(\gamma_1) < \arg Z(\gamma_0)\bigr\}$ is obtained by applying an autoequivalence of $D^b(X)$. The key observation in the proof of the above theorem was taken from \cite{DL22}, which is our Lemma \ref{g_act}: one can identify the action of the autoequivalence $\mc{T}$ on the stability conditions in $\mc{U}^n(\ca)^{\Phi}_+$ with the action of $\wtil{g}$ where $\wtil{g} \in \wtil{\GL}^+(2,\mb{R})$, the universal covering space of $\GL^+(2,\mb{R})$. The former autoequivalence $\mc{T}$  plays an important role in the tilting process (Theorem \ref{double_tilts}).

Recently, Bridgeland-Del Monte-Giovenzana \cite{BDG24} use another method to prove the result in this paper: consider the quiver $Q'$  
\begin{center}
		\begin{tikzcd}
		 0 \arrow[rr, bend left, "x_1" description] \arrow[rr,shift left = 3, bend left, "y_1" description] & & 1 \arrow[ll,  bend left, "x_2" description] \arrow[ll, shift left = 3, bend left, "y_2" description]
		\end{tikzcd}
\end{center}
 and potential $W' = y_1x_2x_1y_2-x_1x_2y_1y_2$, then the Jacobi algebra $ J(Q',W')$ is derived equivalent to the resolved conifold $Y =\co_{\mb{P}^1}(-1)^{\oplus 2}$. $\Phi$ acts as rotation by a half turn on the quiver of local $\mb{F}_0$ (see diagram (\ref{qui:local_f0})), and $(Q',W')$ arises naturally as the quotient. Then we apply the result of Qiu-Zhang \cite{QZ25} (see also \cite{De23}), the invariant subspace of stability conditions is identified with the  stability conditions on $D^b(J(Q',W'))$, where the stability conditions on the resolved conifold is already known \cite{Toda08}. 
 \section*{Notation and Conventions}
\begin{center}
\begin{tabular}{p{4cm} p{10cm}}

$\CD$ & Essentially small triangulated category. \\
$D^b(X)$ & Bounded derived category of coherent sheaves on a noetherian and separated scheme $X$ over $\mb{C}$.\\
$D^b(A)$ & Bounded derived category of right $A$-modules over a noetherian (possibly graded) $\mb{C}$-algebra $A$. \\
$\mathrm{mod}_{0}$-$A$ & Category of nilpotent modules over a positively graded noetherian algebra over $\mb{C}$, where a module $M$ is said to be nilpotent, if there exists $n>0$ such that $M.A_n = 0$.\\
$D^b_0(A)$ & Full subcategory of $D^b(A)$ with complexes having nilpotent cohomology modules.\\
$K_0(\CD)$ (resp. $K_0(\ca)$). & Grothendieck group of an triangulated category $\CD$ (resp. an abelian category $\ca$). In particular $K_0(X) := K_0(\Coh X)$ for variety $X$.\\
$\supp(F)$ & Support of a complex of sheaves $F\in D^b(X)$.\\
$\thick(T)$ & Smallest thick subcategory containing the object $T$ (or set of objects) in $\CD$.\\
$\rep Q$ & Category of finite dimensional representations of a quiver $Q$. 
\end{tabular}
\end{center}

Given a  triangulated category $\mc{D}$, we write 
\[
    \Hom_{\CD}^i(A,B):= \Hom_{\CD}(A,B[i]),
\]
for $A,\ B\in \CD$. We denote by $\Hom^{\bullet}_{\CD}(A,B) = \bigoplus_{i\in\mb{Z}}\Hom_{\CD}(A,B[i])$ the total Hom-space.  

Let $Q = (Q_0,Q_1)$ be a quiver specified by a set of vertices $Q_0$, a set of arrows $Q_1$, and source and target maps $s,\ t:Q_1\rightarrow Q_0$. We compose the arrows  {\bf on the left}, that is for $b,\ a\in Q_1$,  $ba = 0$ unless $s(b)=t(a)$.

Denote by $\mb{C}Q$ the path algebra of $Q$, and given a two-sided $I\subset \mb{C}Q$  generated by linear combinations of paths of length at least $2$, let $A= A(Q,I) = \mb{C}Q/I$. We write $\rep (Q,I) =\mr{mod}_{\text{fd}}$-$A(Q,I)$ and $\rep_{\nil}(Q,I) = \mr{mod}_0$-$A(Q,I)$. For each vertex $i\in Q_0$ there is an associated one-dimensional  simple module $S_i \in\rep(Q,I)$. Note that we have
\[
    n_{ij} = \dim_{\mb{C}}\Ext^1_{A}(S_j,S_i)
\]
where $n_{ij}$ is the number of arrows from vertex $i$ to $j$ in our notations.

\section{Preliminaries}
This section is a summary of the results about the tilting theory in the sense of Happel-Reiten-Smal{\o}\cite{HRS96}, exceptional collections \cite{B90} and stability conditions in \cite{Bri07,Bri08K3}. In this  section let $\CD$ be a $\mb{C}$-linear triangulated category of finite type. The finite type condition is the statement that for any two objects $A$, $B$ of $\CD$ the vector space
	$$
		\Hom^{\bullet}_{\CD}(A,B)=\bigoplus_{i\in \mathbb{Z}}\Hom^i_{\CD} (A,B)
	$$
	is finite-dimensional.

\subsection{Simple tilts} \label{sec:sim_tilt}
The reader is assumed to be familiar with the concept of a t-structure \cite{GM}. We are only considering the bounded t-structures. Recall that a t-structure $(\CD^{\leq 0},\CD^{\geq 0})$ is bounded in $\CD$, if for every object $E\in \CD$, there exists an integer $n>0$ such that $E[n] \in \CD^{\leq 0}$ and $E[-n]\in \CD^{\geq 0}$. The bounded t-structure is determined by its heart:

\begin{lem}[{\cite[Lemma 3.2]{Bri07}}]\label{lem:heart_t}
    Let $\ca \subset \CD$ be a full additive subcategory of $\CD$. Then $\ca$ is the heart of a bounded t-structure $(\CD^{\leq 0}, \CD^{\geq 0})$ if and only if it satisfies the following conditions:
    \begin{enumerate}
        \item if $n_1>n_2$ then $\Hom_{\CD}(A[n_1],B[n_2]) = 0$ for any $A,\ B\in \ca$;
        \item for every nonzero object $E\in \CD$ there are a finite sequence of integers:
        \[
            k_1 > k_2 > \cdots > k_n
        \]
        and a collection of triangles 
    \begin{equation}\label{eqn:heart_t}
        \xymatrix@C=.3em{
            0_{\ } \ar@{=}[r] & E_0 \ar[rrrr] &&&& E_1 \ar[rrrr] \ar[dll] &&&& E_2
                \ar[rr] \ar[dll] && \ldots \ar[rr] && E_{n-1}
            \ar[rrrr] &&&& E_n \ar[dll] \ar@{=}[r] &  E_{\ } \\
            &&& A_1 \ar@{-->}[ull] &&&& A_2 \ar@{-->}[ull] &&&&&&&& A_n \ar@{-->}[ull] 
     }
    \end{equation}
    with $A_i \in \ca[k_i]$ for all $i$.
    \end{enumerate}

\end{lem}

Given $\ca$ the heart of a bounded t-structure and any nonzero object $E$, we denote by $\HO^i_{\ca}(E) := A_i$  the $i$th-graded cohomology group with respect to $\ca$, where $A_i$ appears in (\ref{eqn:heart_t}).

A heart of some t-structure will be called finite-length if it is artinian and noetherian as an abelian category.

The following definition comes from Happel-Reiten-Smal{\o} \cite{HRS96}.
\begin{defi}[Torsion pair]
Let $\ca$ be a heart of some bounded t-structure in the triangulated category $\mc{D}$. A pair of full subcategories $(\mc{T},\mc{F})$ of $\ca$ is called a torsion pair in $\ca$ if it satisfies the following conditions
   \begin{enumerate}
       \item $\Hom_{\ca}(T,F) = 0$ for $T\in \mc{T}$ and $F\in \mc{F}$;
       \item for any object $A\in \ca$, there exist $M\in \mc{T}$ and $N\in \mc{F}$ such that they fit into a short exact sequence
       $$
        \xymatrix{
         0 \ar[r] & M \ar[r] & A \ar[r] & N \ar[r] & 0.
        }
       $$
   \end{enumerate}
\end{defi}

 The following theorem was proved in \cite[Proposition 2.1]{HRS96}.
 \begin{thm}[Happel-Reiten-Smal{\o}]
Let $(\mc{T},\mc{F})$ be a torsion pair in a heart $\ca$. Let
    \begin{eqnarray*}
        \ca^{\sharp} &:=& \left\{ E\in \CD\ | \ \ho^1_{\ca}(E)\in \mc{T},\ \ho^0_{\ca}(E)\in \mc{F}, \ \ho^i_{\ca}(E) = 0 \text{ for }i\neq0,\ 1 \right\}, \\
        \ca^{\flat} &:=& \left\{ E\in \CD \ | \ \ho^{-1}_{\ca}(E) \in \mc{F},\ \ho^0_{\ca}(E)\in \mc{T},\ \ho^i_{\ca}(E) = 0 \text{ for }i\neq-1,\ 0 \right\},
    \end{eqnarray*}
    then  $\ca^{\sharp}$ and $\ca^{\flat}$ are hearts of  bounded t-structures in $\mc{D}$.
 \end{thm}
 A special case of the tilting construction will be particularly important \cite[Definition 3.7]{KQ15}. Suppose that $\ca$ is a finite-length heart and $S\in \ca$ is a simple object. Let $\langle S\rangle$ be the full subcategory consisting of objects $E\in \ca$ all of whose simple factors are isomorphic to $S$. Define the full subcategories 
 $$
    S^{\perp} := \{E\in \ca \ | \ \Hom_{\ca}(S, E) = 0\}, \quad {}^{\perp}S:= \{E\in\ca \ |\ \Hom_{\ca}(E, S) = 0\}.
 $$
 Then we can either view $(\langle S \rangle, S^{\perp})$  or $(^{\perp} S, \langle S \rangle )$ as a torsion pair. Then we can define new tilted hearts
 \begin{equation} \label{rltilt}
     L_S \ca := \langle S[1], ^{\perp} S \rangle, \quad R_S \ca := \langle S^{\perp}, S[-1] \rangle,
 \end{equation}
 which we refer to as the left and right simple tilts of the heart $\ca$ at the simple object $S$. 
 \begin{rem} \label{inv_tilt}
  It is easy to see that $S[-1]$ is a simple object of $R_S \ca$ and that if the category is of finite-length, then $L_{S[-1]} R_S \ca = \ca$. Similarly, if $L_S \ca$ is of finite-length then $R_{S[1]} L_S \ca = \ca$. 
 \end{rem}

 The following lemmas will be useful.
 \begin{lem}\label{nest_t}
     Let $(D^{\leq 0}, D^{\geq 0})$ and $(\widetilde{D}^{\leq 0}, \widetilde{D}^{\geq 0})$ be two bounded t-structures of $\CD$, and we denote by $\ca$ and $\ca'$  their hearts respectively. If $\ca \subset \ca'$ then $\ca = \ca'$.
 \end{lem}
 \begin{proof}
     Let $E\in \ca'$. Since for any object $F\in D^{\leq 0}$, $F$ has a finite filtration by objects in $\ca[k_i] \subset \ca'[k_i]$ for $k_i \geq 0$ by Lemma \ref{lem:heart_t}, so we have
     \[\Hom_{\CD}(F,E[-1]) = 0.\] 
     Therefore $E[-1]\in D^{>0}$.
     
     Similarly for any object $G\in D^{>0}$, $G$ has a finite filtration by objects in $\ca[k_i]\subset \ca'[k_i]$ for $k_i < 0$, therefore 
     \[
        \Hom_{\CD}(E, G) = 0.
     \]
     Therefore $E\in D^{\leq 0}$. So we have $E\in D^{\leq 0}\cap D^{>0}[1] = \ca$. This proves the lemma.
 \end{proof}
 \begin{lem}\label{aut_sim_tilt}
    Take an autoequivalence $\Phi\in\Aut (\mc{D})$. Let $\ca\subset \mc{D}$ be a heart of some bounded t-structure and of finite-length, $S\in \ca$ be a simple object. Then we have 
    $$
        \Phi(L_S\ca) = L_{\Phi(S)} \Phi(\ca), \quad \Phi(R_S\ca) = R_{\Phi(S)} \Phi(\ca).
    $$
 \end{lem}
\begin{proof}
By the definition of simple tilts (\ref{rltilt}), we have $ R_{\Phi(S)}\Phi(\ca) = \left\langle \left(\Phi(S)\right)^{\perp}, \Phi(S)[-1]\right\rangle$. It is easy to check $\Phi(S^{\perp}) = (\Phi(S))^{\perp}$, therefore $R_{\Phi(S)}\Phi(\ca) \subset \Phi(R_S\ca)$ by definition. By Lemma \ref{nest_t} we have $R_{\Phi(S)}\Phi(\ca) = \Phi(R_S \ca)$. The proof of the left tilt case is similar. 
\end{proof}
Given a heart of bounded t-structure $\ca\subset \CD$, we denote by $\Sim \ca$ the  set of all non-isomorphic simple objects in $\ca$. The following theorem  characterizes the new simple objects in the tilted hearts.
\begin{prop}[{\cite[Proposition 5.4]{KQ15}}]\label{sim_tilt}
    Assume $\Sim \ca$ is finite and $\ca$ is of finite-length. Let $S\in \Sim \ca$ be such that $\Ext_{\ca}^1(S,S) = 0$. Then after taking a left or right simple tilt, the new simple objects are:
    \begin{eqnarray}
        \Sim R_S\ca &=& \{ S[-1]\} \ \cup \ \{\phi_S(X): X\in \Sim \ca,\ X\neq S\} \\
        \Sim L_S\ca &=& \{S[1] \} \ \cup \ \{\psi_S(X): X\in \Sim \ca,\ X\neq S\}
    \end{eqnarray}
    where 
    \begin{eqnarray*}
        \phi_S(X) &=& \Cone\ \left(S[-1]\otimes \Ext^1(S,X) \longrightarrow X\right),\\
        \psi_S(X) &=& \Cone\ \left(X\longrightarrow S[1]\otimes \Ext^1(X,S)^*\right)[-1].
    \end{eqnarray*}
\end{prop}
On the other hand, the concept of tilting objects gives another method to construct new bounded t-structures. Recall that an object $T$  in $\CD$ is called a tilting object if it satisfies the following conditions
\begin{enumerate}
    \item $\RHom_{\CD}\left(T,T[n]\right) = 0$, unless $n=0$;
    \item $T$ is  a classical generator of $\CD$, i.e., the smallest thick subcategory containing $T$ which we denote by $\thick(T)$ is $\CD$.
\end{enumerate}
Suppose $X$ is a smooth quasi-projective variety and $\CD = D^b(X)$. Then $\RHom_X(T,-)$ induces a derived equivalence between $X$ and noncommutative algebra $B=\End_X(T)$ \cite{R89, Keller94}:
\[
    \RHom_X(T,-): D^b(X) \longrightarrow D^b(B).
\]
Pulling back the standard t-structure on $D^b(B)$ via the equivalence gives us a new t-structure on $D^b(X)$. 
\subsection{Exceptional collections}
Usually tilting objects in $D^b(X)$ can break up into small pieces called the exceptional objects. 
\begin{defi}[Exceptional collection] \label{exc_coll}
An object $E$ in $\CD$ is said to be exceptional if 
$$
	\Hom^k_{\CD}(E,E) = \left\{ \begin{array}{cc}
								\BC&\  \text{if } k=0,\\
								0 &\  \text{otherwise}.
								\end{array}
							\right.
$$
An exceptional collection ${\BE}\subset \CD$ is a sequence of exceptional objects 
$$
	{\BE}=(E_0,\cdots,E_n)
$$
such that for all $0\leq i<j\leq n$, we have $\Hom^{\bullet}_{\CD}(E_j,E_i) = 0$.
\end{defi}
An exceptional collection $\BE = (E_0,\cdots, E_n)$ is said to be strong if for all $i,\ j$ 
$$
	\Hom^k_{\CD}(E_i,E_j) = 0,\ \ \text{unless } k=0.
$$

We write $\thick(\BE) \subset \CD$ for the smallest thick  subcategory of $\CD$ containing the elements of an exceptional collection $\BE\subset \CD$. An exceptional  collection $\BE$ is said to be full if $\thick(\BE) = \CD$. From the definitions above, we have that for a full and strong exceptional collection $\BE$,  the object $\bigoplus_{i=0}^{n} E_i$ is a tilting object in $\CD$. The first full and strong exceptional collection was found in $D^b(\mb{P}^N)$ by Beilinson \cite{Bei78}.
\begin{exa}
    $D^b(\mb{P}^N)$ admits a full and strong exceptional collection $\bigl(\co,\co(1),\cdots, \co(N)\bigr)$. $\mb{P}^N$ is derived equivalent to the path algebra of quiver
    \[ 
    \begin{tikzcd}
   0 \arrow[r, shift left=1, phantom, "\vdots" description] \arrow[r, shift left=5, "f_1" description] \arrow[r, shift right=5, "f_N" description] &
   1 \arrow[r, shift left=1, phantom, "\vdots" description] \arrow[r, shift left=5, "f_1" description] \arrow[r, shift right=5, "f_N" description] &
   \cdots \arrow[r, shift left=1, phantom, "\vdots" description] \arrow[r, shift left=5, "f_1" description] \arrow[r, shift right=5, "f_N" description] &
   N-1 \arrow[r, shift left=1, phantom, "\vdots" description] \arrow[r, shift left=5, "f_1" description] \arrow[r, shift right=5, "f_N" description] &
   N
    \end{tikzcd} 
\]
subject to the relations $f_{j+1} f_j = f_j f_{j+1}$.
\end{exa}

Given an exceptional collection $\BE$ in  $\CD$, the right orthogonal subcategory to $\BE$ is the full triangulated subcategory 
$$
	\BE^{\bot}=\left\{ X\in \CD: \ \Hom^{\bullet}_{\CD}(E,X) = 0\ \text{for } E\in \BE\right\}.
$$
Similarly, the left orthogonal subcategory to $\BE$ is 
$$
	{}^{\bot} \BE = \left\{ X\in \CD: \ \Hom^{\bullet}_{\CD}(X,E) = 0\ \text{for } E\in \BE\right\}.
$$
The subcategory $\langle \BE\rangle$ is admissible due to \cite[Theorem 3.2]{B89}, i.e. the inclusion functor $i: \langle \BE \rangle \rightarrow \CD$ has left and right adjoint functors. Thus the fullness of $\BE$ is equivalent to $\BE^{\bot} = 0$ or ${}^{\bot} \BE = 0$.

We suppose $E\in \CD$ to be exceptional. Given an object $X\in \CD$, the left mutation of $X$ through $E$ is the object $L_E(X)$ defined up to isomorphism by the triangle
$$
	\xymatrix{
		L_E(X) \ar[r]& \Hom^{\bullet}_{\CD}(E,X)\otimes E \ar[r]^-{ev} & X \ar[r]&L_E(X)[1],
		}
$$
where $ev$ denotes the evaluation map. Similarly, given $X\in \CD$, the right mutation of $X$ through $E$ is the object $R_E X$ defined by the triangle
$$
	\xymatrix{
     X \ar[r]^-{coev} & \Hom^{\bullet}_{\CD}(X,E)^* \otimes E \ar[r] &	R_E(X) \ar[r] &  X[1],
	}
$$
where $coev$ denotes the coevaluation map. Moreover, consider the left and right orthogonal subcategories of $E$, these two operations define mutually inverse equivalences of categories (see \cite[Appendix B]{BriS10})
\begin{equation}\label{eqn:lr_mut_equ}
    \xymatrix{
		{}^{\bot}E \ar@/^/[rr]^{L_E} & & E^{\bot} \ar@/^/[ll]^{R_E}
	}
\end{equation}
	
\begin{defi}[Standard mutation]
Given a full exceptional collection $\BE = (E_0,\cdots, E_n)$, the mutation operation $\sigma_i$ for each $0<i\leq n$ is defined by the rule
\begin{eqnarray*}
	&&\sigma_i(E_0,\cdots, E_{i-2}, E_{i-1},E_i, E_{i+1},\cdots, E_n) \\
	&=&(E_0,\cdots, E_{i-2}, L_{E_{i-1}}(E_i),E_{i-1}, E_{i+1},\cdots, E_n)
\end{eqnarray*}
\end{defi}
This operation takes exceptional collections to exceptional collections \cite[Lemma 2.1]{B89}. And it takes full collections to full collections \cite[Lemma 2.2]{B89}.

The following definition is due to Bondal\cite{B89}, and we refer our reader to \cite[Appendix B]{BriS10} for the proof of (\ref{eqn:dual}).
\begin{defi}[Dual objects] \label{dual_coll}
Let $\BE = (E_0,\cdots, E_n)$ be a full exceptional  collection and define 
$$
	F_j = L_{E_0}L_{E_1} \cdots L_{E_{j-1}} (E_j)[j],\ \ 0\leq j \leq n.
$$
Then $F_j$ is called the dual object to $E_j$ and satisfies 
\begin{equation}\label{eqn:dual}
	\Hom^k_{\CD}(E_i,F_j) = \left\{ \begin{array}{cc}
							\BC & \text{if}\ i=j \ \text{and} \ k=0,\\
							0 & \text{otherwise.}
						\end{array}
						\right.
\end{equation}
\end{defi}

\subsection{Stability conditions}
We collect some properties and theorems on the space of stability conditions introduced in \cite{Bri07}.
\begin{defi}[Slicing]\label{defi:slicing}
A slicing of $\CD$ is a collection of full subcategories $\mc{P}(\phi)$ indexed by $\phi\in \mb{R}$, satisfying the following axioms:
\begin{enumerate}
    \item $\mc{P}(\phi+1) = \mc{P}(\phi)[1]$;
    \item $\Hom_D\left(\mc{P}(\phi_1), \mc{P}(\phi_2)\right) = 0$ for $\phi_1 > \phi_2$;
    \item for any nonzero object $E\in \CD$, we have a collection of triangles 
    \[
        \xymatrix@C=.5em{
            0_{\ } \ar@{=}[r] & E_0 \ar[rrrr] &&&& E_1 \ar[rrrr] \ar[dll] &&&& E_2
                \ar[rr] \ar[dll] && \ldots \ar[rr] && E_{n-1}
            \ar[rrrr] &&&& E_n \ar[dll] \ar@{=}[r] &  E_{\ } \\
            &&& A_1 \ar@{-->}[ull] &&&& A_2 \ar@{-->}[ull] &&&&&&&& A_n \ar@{-->}[ull] 
     }
    \]
    such that $A_i\in \mc{P}(\phi_i)$, and 
    \[
        \phi_1 > \phi_2 >\cdots > \phi_n.  
    \]
\end{enumerate}
For any nonzero object $E$, we denote by $\phi^+(E) = \phi_1$ and $\phi^-(E) = \phi_n$ where $\phi_i$ is defined as above. For any interval $I\subset \mb{R}$, $\mc{P}(I)$ is defined to be the extension-closed subcategory of $\CD$ generated by objects $E\in \mc{P}(\phi)$ for $\phi \in I$.
\end{defi}
We denote  by $\mr{Slice}(\CD)$ the set of all slicings on $\CD$. Bridgeland introduced a generalized metric in $\mr{Slice}(\CD)$:
\begin{defi}[{\cite[Section 6]{Bri07} }]\label{defi:slice_metric}
Let $\mc{P}_1$, $\mc{P}_2\in \mr{Slice}(\CD)$, then the generalized metric $d: \mr{Slice}(\CD)\times \mr{Slice}(\CD) \rightarrow [0,+\infty]$ is defined as  
\[
    d(\mc{P}_1,\mc{P}_2) := \mathop{\mr{sup}}_{E\neq 0\in \CD}\left\{|\phi_1^+(E)-\phi_2^+(E)|, |\phi_1^-(E)-\phi_2^-(E)|\right\}.
\]
\end{defi}
Before recalling stability condition on $\CD$, we first recall the stability function on an abelian category $\ca$ \cite{Ru97}.
\begin{defi}\label{def:stab_func}
A stability function on $\ca$ is a group homomorphism $Z:K_0(\ca)\rightarrow \mb{C}$ such that for any nonzero object $A\in\ca$, the complex number $Z(A)$ lies in the subset 
\[
    H = \{z = r\mr{exp}(i\pi\phi) \vert r>0, 0 <\phi\leq 1\} \subset \mb{C}.
\]
The phase of $A$ is defined to be $\phi(A) = \frac{1}{\pi} \mr{arg}Z(A) \in (0,1]$. An object $E\in \ca$ is said to be (semi)stable if for any subobject $A\subset E$ we have
\[
    \phi(A) < (\leq) \phi(E).
\]
\end{defi}
\begin{defi}[Harder-Narasimhan property{\cite[Definition 2.3]{Bri07}}] \label{defi:HN}
Let $Z:K_0(\ca)\rightarrow \mb{C}$ be a stability function on the abelian category $\ca$. Then $Z$ is said to have Harder-Narasimhan property if for any nonzero object $E\in \ca$ there is a filtration
\[
    0 = E_0 \subset E_1\subset E_2\subset \cdots E_{n-1}\subset E_n = E
\]
such that each $F_i = E_i/E_{i-1}$ is a semistable object of phase $\phi_i$ and  $\phi_1 > \phi_2\cdots >\phi_{n-1}>\phi_n$.
\end{defi}

\begin{defi}[Stability condition]
A stability condition for $\CD$ is a pair $\sigma = (Z, \ca)$ which consists of a heart of a bounded t-structure $\ca$ in $\CD$, and a stability function (called the central charge of $\sigma$) $Z: K_0(\ca) \rightarrow \mb{C}$ such that $Z$ satisfies the Harder-Narasimhan property.
\end{defi}
The above definition of stability condition is equivalent to the following definition \cite[Proposition 5.3]{Bri07}:
\begin{defi}
A stability condition is a pair $\sigma = (Z,\mc{P})$ which consists of a slicing $\mc{P}\in \mr{Slice}(\CD)$ and a group homormorphism called the central charge $Z: K_0(D) \rightarrow \mb{C}$, such that it satisfies the compatibility condition: if $0\neq E\in \mc{P}(\phi)$ for some $\phi\in \mb{R}$, then 
\[
    Z(E)  = r\mr{exp}(i\pi\phi),\ r>0.
\]
The objects in $\mc{P}(\phi)$ are called semistable of phase $\phi$, and the simple objects in $\mc{P}(\phi)$ are called stable.
\end{defi}

The following lemma will be useful later.
\begin{lem}[{\cite[Lemma 6.4]{Bri07}}]\label{stab_equi}
    If the stability conditions $\sigma = (Z, \mc{P})$ and $\tau = (Z, \mc{P}')$ have the same central charge and $d(\mc{P},\mc{P}') < 1$, then $\sigma = \tau$.
\end{lem}

\begin{defi}[Support property]
Let $\sigma = (Z,\mc{P})$ be a stability condition, by fixing a norm $\lVert\ \cdot\ \rVert$ on $K_0(\CD)_{\mb{R}} = K_0(\CD)\otimes_{\mb{Z}}\mb{R}$, $\sigma$ is said to have support property if there exists a constant $C>0$ such that 
$$
    \lVert E \rVert \leq C |Z(E)|
$$
for any stable object $E$.
\end{defi}
We denote  by $\Stab(\CD)$ the set of all stability conditions with the support property.
To define the topology on $\Stab(\CD)$, Bridgeland \cite{Bri07} introduced the following definitions:
\begin{defi}
Let $\sigma = (Z, \mc{P})\in \Stab(\CD)$. The function  $\lVert\ \cdot\ \rVert_{\sigma}: \Hom(K_0(\mc{D}), \mb{C}) \rightarrow [0,+\infty]$ is defined as
\[
    \lVert W \rVert_{\sigma} := \sup \left\{ \frac{|W(E)|}{|Z(E)|} : \text{ E semistable for } \sigma \right\}.
\]
\end{defi}
\begin{lem}[{\cite[Lemma 6.2]{Bri07}}]
For $\sigma = (Z,\mc{P})\in \Stab(\CD)$ and $0<\epsilon < \frac{1}{4}$ let
$$
    C_{\epsilon}(\sigma) := \left\{ \tau = (W, \mc{Q})\in \Stab(\CD): \lVert W - Z \rVert_{\sigma} < \sin(\pi\epsilon), d(\mc{P}, \mc{Q}) < \epsilon \right\}.
$$
Then by varying $\sigma$,  we get a basis for the topology of $\Stab(\CD)$. 
\end{lem}

The following is the main result of \cite{Bri07}. The idea is that if $\sigma = (Z,\CP)$ is a stability condition on $\CD$ and one deforms $Z$ to a new group homomorphism $W:K_0(\CD) \rightarrow \BC$ in such a way that the phase of each semistable object in $\sigma$ changes in a uniformly bounded way, then it is possible to define a new slicing $\mc{Q}(\psi)\subset \CD$ so that $(W,\mc{Q})$ is a stability condition on $\CD$.

\begin{thm}[Deformation of stability conditions] \label{stab_deform}
Let $\sigma = (Z, \mc{P})\in \Stab(\CD)$. Then there exists $0<\epsilon_0 <\frac{1}{8}$ such that for $0<\epsilon < \epsilon_0$ and $W\in \Hom(K_0(\CD), \mb{C})$ satisfying
$$
    |W(E)-Z(E)| < \sin(\pi \epsilon)|Z(E)|
$$
for any semistable object $E\in \CD$ with respect to $\sigma$,  there exists a unique stability condition $\tau = (W, \mc{P}')$ such that 
$$
    d(\mc{P},\mc{P}') < \epsilon.
$$
\end{thm}
\begin{cor}[{\cite[Theorem 1.2]{Bri07}}] \label{stab_mani}
Let $\CD$ be a triangulated category. For each connected component $\Sigma \subset \Stab(\CD)$ there are a linear subspace $V(\Sigma) \subset \Hom(K_0(\CD),\mb{C})$, with a well-defined linear topology, and a local homeomorphism $\mc{Z}: \Sigma \rightarrow V(\Sigma)$ which sends a stability condition to its central charge $Z$.
\end{cor}
 Since $K_0(\CD)$ might have infinite rank, in practice we usually assume there is a quotient group $\mc{N}$ of finite rank, and the quotient map is denoted by $\mu: K_0(\CD) \rightarrow \mc{N}$. Then let $\Stab_{\mc{N}}(\CD)$ be the subspace of $\Stab(\CD)$ consisting of stability conditions whose central charges $Z: K_0(\CD)\rightarrow \mb{C}$ factor through $\mc{N}$. Then the following result is an immediate consequence of Corollary \ref{stab_mani}.
\begin{cor}[{\cite[Corollary 1.3]{Bri07}}]\label{cor:stab_mani_2}
    For each connected component $\Sigma \subset \Stab_{\mc{N}}(\CD)$ there are a linear subspace $V(\Sigma) \subset \Hom(\mc{N},\mb{C})$, and a local homeomorphism $\mc{Z}: \Sigma \rightarrow V(\Sigma)$ which sends a stability condition to its central charge $Z$. In particular, $\Sigma$ is a finite-dimensional complex manifold.
\end{cor}

We recall some group actions on the space of stability conditions. Let $\widetilde{\mr{GL}}^+(2,\mb{R})$ be the universal covering of $\mr{GL}^+(2,\mb{R})$. Note that an element in $\widetilde{\mr{GL}}^+(2,\mb{R})$ can be viewed as a pair $(g,f)$ where $g\in \mr{GL}^+(2,\mb{R})$ and $f:\mb{R}\rightarrow \mb{R}$ is an increasing map with $f(\phi+1) = f(\phi) +1$, such that $g$ and $f$ induce the same action on the circle $S^1 = \{e^{i\pi\phi} : \phi \in \mb{R}\} =(\mb{R}^2\setminus\{0\})/\mb{R}_{>0}$.
\begin{defi}\label{defi:group_act}
    The space of stability conditions carries a right action by $\widetilde{\mr{GL}}^+(2,\mb{R})$. For $\widetilde{g} = (g,f)\in \widetilde{\mr{GL}}^+(2,\mb{R})$ and $\sigma = (Z,\mc{P}) \in \Stab(\CD)$, then $\sigma\cdot \widetilde{g} = (Z_g, \mc{P}_f)$ where for $[E] \in K_0(\CD)$
    $$
        Z_g(E) = g^{-1}Z(E), \quad \mc{P}_f(\phi) = \mc{P}\left(f(\phi)\right).
    $$ 
    The space of stability conditions also carries a left action by $\Aut(\CD)$. For $T\in \Aut(D)$,  denote by $t$ the automorphism of $K_0(\CD)$ induced by $T$, then $T(\sigma) = (Z_t, \mc{P}_T)$ where for $[E]\in K_0(\CD)$
    $$
        Z_t(E) = Z(t^{-1}E), \quad \mc{P}_T(\phi) = T\left(\mc{P}(\phi)\right).
    $$
\end{defi}
\begin{rem}\label{rem:C_act}
    From the definition of $\widetilde{\mr{GL}}^+(2,\mb{R})$-action, $\sigma$ and $\sigma\cdot\widetilde{g}$ have the same set of semistable objects, but the phases have been relabelled. In particular, note that the additive group $\mb{C}$ acts on $\Stab(\CD)$, via the embedding $\mb{C}\xhookrightarrow{} \widetilde{\mr{GL}}^+(2,\mb{R})$: an element $\lambda \in\mb{C}$ acts by
    \[
        \lambda: (Z,\mc{P}) \mapsto (Z',\mc{P}'),\quad Z'(E) = e^{-i\pi\lambda} \cdot Z(E), \quad \mc{P}'(\phi)  = \mc{P}\left(\phi+\mr{Re}(\lambda)\right).
    \]
\end{rem}

In the end of this subsection, we recall the following important lemma:
\begin{lem}[{\cite[Proposition 7.6]{BI15}}]\label{stab_sstable}
    Fix $0\neq E\in D$. Then 
    \begin{enumerate}
        \item the set of stability conditions $\sigma \in \Stab(\CD)$ for which $E$ is $\sigma$-stable is open;
        \item the set of stability conditions $\sigma \in \Stab(\CD)$ for which $E$ is $\sigma$-semistable is closed.
    \end{enumerate}
\end{lem}
\begin{defi}
    Let $\ca$ be the heart of  a bounded t-structure in $\CD$ which is of finite-length. Then the subset of stability conditions $\mc{U}(\ca)\subset \Stab(\CD)$  is defined to be
    \[
    \mc{U}(\ca) = \bigl\{\sigma= (Z,\mc{P})\vert \mc{P}\bigl((0,1]\bigr) = \ca \bigr\}.
    \]
\end{defi}
The relation between simple tilts and stability conditions is the following:
\begin{lem}[{\cite[Lemma 5.5]{Bri06}}]\label{lem:stab_tilt}
Suppose $\ca$ is of finite-length. Let $\sigma = (\mc{Z},\mc{P}) \in \overline{\mc{U}}(\ca)$ the closure of $\mc{U}(\ca)$. Suppose that $Z(S_i)\in \mb{R}_{<0}$ for some $i$, also $\mr{Im}Z(S_j)>0$ for $j\neq i$, and $R_{S_i}\ca$ is finite length, then there is an open neighborhood  $V$ of $\sigma$ such that $V \subset \mc{U}(\ca)\cup \mc{U}(R_{S_i}\ca)$. Similarly suppose $Z(S_i)\in \mb{R}_{>0}$ for some $i$ and $\mr{Im}Z(S_j)>0$ for $j\neq i$, and $L_{S_i}\ca$ is finite length, then there is an open neighborhood $V'$ of $\sigma$ such that $V' \subset \mc{U}(\ca)\cup \mc{U}(L_{S_i}\ca)$.
\end{lem}

\section{Quiver symmetry and autoequivalence} \label{sec:quiv_symm_aut}
In this section, we give explicit construction of the autoequivalence in the introduction. 
\subsection{Quiver}
There is a full and strong exceptional collection on $\ff = \pp$:
\[
    \mb{E} = \left(\co(0,0), \co(1,0),  \co(1,1), \co(2,1)\right),
\]
which has the dual collection
\[
    \mb{F} = \left(\co(0,0), \co(-1,0)[1],\co(1,-1)[1],\co(0,-1)[2]\right).
\]
As in the introduction, we denote by $\pi: X \rightarrow \ff$ the bundle projection map, $p_i: \pp \rightarrow \mb{P}^1$, $i=1,\ 2$ the projection maps to each component. We denote by $s: \ff \hookrightarrow X$ the embedding map  of the zero section.

\begin{lem} \label{lem:van_tilt_f0}
The pull back $\mc{Q}= \bigoplus_i \pi^*E_i$ is a tilting object in $D^b(X)$.
\end{lem}
\begin{proof}
    For $a,\ b,\ c,\ d\in \BZ$, we have
    \begin{eqnarray*}
        \Ext^i_X\left(\pi^*\co(a,b), \pi^*\co(c,d)\right) &=& \Ext^i_{\ff}\left(\co(a,b), \pi_*\pi^*\co(c,d)\right) \\
                                                    &=& \bigoplus_{n=0}\Ext^i_{\ff}\left(\co(a,b), \co(c,d)\otimes (\omega^*_{\ff})^n\right) \\
                                                    &=& \bigoplus_{n=0} \ho^i\left(\ff, \co(c-a+2n,d-b+2n)\right)\\
                                                    &=& \bigoplus_{n=0}\bigoplus_{s+t = i}\ho^s(\mb{P}^1,\co(c-a+2n))\otimes \ho^t(\mb{P}^1,\co(d-b+2n)).
    \end{eqnarray*}
    For $i>0$,  $\ho^i\bigl(\ff, \co(c-a+2n,d-b+2n)\bigr) = 0$ unless 
    \begin{enumerate}
        \item $c-a+2n \leq -2, \ d-b+2n \geq 0$;
        \item  $d-b+2n \leq -2, \ c-a+2n \geq 0$;
        \item $c-a+2n\leq -2, \ d-b+2n \leq -2$.
    \end{enumerate}
    Since $-1\leq d-b\leq 1$ in our case, so only the first part of case 1 is possible. One can easily verify  the bundles in our exceptional collection do not belong to this case. Therefore $\Ext^i_{X}(\mc{Q},\mc{Q}) = 0$ for $i\neq 0$.

    The proof of the generating property for $\mc{Q}$ is due to the general result \cite[Lemma 5.2.3]{Lam}. We finished the proof.
\end{proof}
The endomorphism algebra $B=\End_X(\mc{Q})$ is noetherian \cite[Theorem 3.6]{BriS10}, therefore we can write it as the path algebra of a quiver $Q$ subject to relations, and  grade it by the length of paths. The vertex is indexed by 
$i$, and the number of arrows from $i$ to $j$ is the dimension of space of irreducible maps from $\pi^*E_i$ to $\pi^*E_j$, i.e., the cokernel of the map
\[
    \bigoplus_{k \neq i,\ j}\Hom_X(\pi^*E_i,\pi^*E_k)\otimes \Hom_X(\pi^*E_k,\pi^*E_j) \longrightarrow \Hom_X(\pi^*E_i, \pi^*E_j).
\]
The following corollary is from \cite[Lemma 4.4]{Bri05}.
\begin{cor}
    There is a derived equivalence:
    $$
        \mc{R}_{\mc{Q}} := \RHom(\mc{Q},-): D^b_0(X) = D^b_0(X) \longrightarrow D^b_0(B).
    $$  
\end{cor}
$D^b_0(B)$ inherits a t-structure from the standard t-structure on $D^b(B)$, whose heart is $\mr{mod}_0$-$B$. From now on until the end of the paper, we denote by $\ca$ the heart of the t-structure by pulling back the standard t-structure on $D^b_0(B)$, i.e., it is equivalent to $\mr{mod}_0$-$B$. 
\begin{cor}\label{simp_heart}
The simple objects up to isomorphism in $\ca$ are 
 \[
 S_0 = s_*\co(0,0), \quad S_1 = s_*\co(-1,0)[1], \quad S_2 = s_* \co(1,-1)[1], \quad S_3 = s_*\co(0,-1)[2].
 \]
\end{cor}
\begin{proof}
 We write $P_i$  the projective  $B$-module and $C_i$ the simple $B$-module associated with vertex $i$. Then $\mr{mod}_0$-$B$ is the extension-closed subcategory of $\mr{mod}$-$B$ generated by $\{C_i\}_{i\in I}$ and $\{C_i\}_{i\in  I}$ is the set of all simple $B$-modules in $\mr{mod}_0$-$B$.  By definition of $\mc{R}_{\mc{Q}}$, $\pi^*E_i$ is sent to $P_i$. Then $s_*F_i$ is sent to $C_i$ which follows from the definition of the dual collection:
    \[
        \Hom^{\bullet}_{D^b(X)}(\pi^*E_i, s_*F_j) = \Hom^{\bullet}_{D^b(\ff)}(E_i,F_j) = \delta_{ij}\mb{C}.
    \]
\end{proof}
\begin{prop}\label{quiver_tilt}
The quiver $Q = (Q_0, Q_1)$ of the endomorphism algebra $\mr{End}_X(\mc{Q})$ is
\begin{equation}\label{eqn:quiver}
    \xymatrixcolsep{4pc}
    \xymatrix{
        0 \ar@<0.5ex>[r]^{x_1} \ar@<-0.5ex>[r]_{y_1} & 1 \ar@<0.5ex>[d]^{x_2} \ar@<-0.5ex>[d]_{y_2} \\
        3 \ar@<0.5ex>[u]^{x_3} \ar@<-0.5ex>[u]_{y_3} & 2 \ar@<0.5ex>[l]^{x_4} \ar@<-0.5ex>[l]_{y_4}
    }
\end{equation}
 and the vertex $i$ corresponds to $S_i$.
\end{prop}
\begin{proof}
    The arrows from $i$ to $j$ can be calculated by the  dimension of the vector space $\Ext^1_X(S_j, S_i)$. By using the Koszul resolution \cite[Chapter 11]{Huy06} along the embedding map $s$, for any sheaf $F$ on $\ff$ we have 
    \[
        s^*s_* F \cong F \oplus \left(F\otimes \omega_{\ff}^*[1]\right)
    \]
    in $D^b(\ff)$. Therefore 
    \[
        \Ext^n_X(s_*F_i, s_*F_j ) = \Ext^n_{\ff}(F_i,F_j) \oplus \Ext^{3-n}_{\ff}(F_j, F_i)^*.
    \]
   For example 
   \begin{eqnarray*}
       &&\Ext^1_X\left(s_*\co(1,-1)[1],s_*\co(-1,0)[1]\right) \\
       &=& \Ext^1_{\ff}\left(\co(1,-1),\co(-1,0)\right) \oplus \Ext^2_{\ff}\left(\co(-1,0),\co(1,-1)\right)^* \\
       &=& \ho^1\left(\pp,\co(-2,1)\right)\\
       &=&\mb{C}^2.
   \end{eqnarray*}
   The other calculations are similar. Therefore we get the given quiver $Q$.
\end{proof}
We denote by $I$ the relations of paths in $\End_X(\mc{Q})$ and $ \rep_{\mr{nil}}(Q,I) \cong \ca$ the category of nilpotent representations of quiver with relations. 
\begin{rem}\label{rem:potential}
    Since $B=\End_X(\mc{Q})$ is graded 3-Calabi-Yau in the sense that the full subcategory consisting of objects with finite dimensional cohomology modules $D^b_{\text{fin}}(B)$ has Serre duality \cite{Kel08}:
    \[
        \Ext^i_{B}(M,N) = \Ext^{3-i}_B(N,M)^*,\quad M,\ N\in D^b_{\text{fin}}(B).
    \]
     Then work of Bocklandt \cite{Bo08} shows that the relations of $\End_X(\mc{Q})$ can be encoded in compact form in a potential. Thus we can write
     \[
        B = B(Q,W) = \mb{C}Q/(\partial_a W: a\in Q_1)
     \]
     for some non-uniquely defined element $W\in \BC Q/[\BC Q, \BC Q]$. In fact we can write down the potential $W$ explicitly here: keep the notations as in (\ref{eqn:quiver}). Then by using Segal's result \cite{Seg08}, we have
     \[
        W=x_4x_3x_2x_1 + y_4y_3y_2y_1-y_4x_3y_2x_1-x_4y_3x_2y_1.
     \]
     The corresponding relations are (for $j\in \BZ_4$):
     \begin{eqnarray*}
         \partial_{x_j}W = x_{j+3}x_{j+2}x_{j+1}-y_{j+3}x_{j+2}y_{j+1} = 0, &\\
         \partial_{y_j}W = y_{j+3}y_{j+2}y_{j+1}-x_{j+3}y_{j+2}x_{j+1}=0. &
     \end{eqnarray*}
\end{rem}
\subsection{Autoequivalence and invariant stability conditions}\label{sec:auto}
The spherical object and spherical twist were introduced by Seidel and Thomas \cite{ST01}. We briefly recall the definition and property here: for our use we simply consider $\CD := D^b(\mc{V})$  where $\mc{V}$ is a local Calabi-Yau variety of dimension $n$.
\begin{defi}
An object $S\in \CD$ is called $n$-spherical if the following conditions are satisfied:
\begin{enumerate}
    \item For any $F\in \CD$, $\Hom^{\bullet}_{\CD}(F,S)$  and $\Hom^{\bullet}_{\CD}(S,F)$ have finite (total) dimension over $\mb{C}$.
    \item We have 
$$
    \Ext^k_{\CD}(S,S) = \left\{ \begin{array}{cc}
                            \mb{C}& k = 0,\ n, \\
                            0 & \text{otherwise.}
                            \end{array}
                            \right.
$$
\end{enumerate}

Let $S$ be a spherical object in  $\CD$, then the spherical twist $\Tw_S(E)$ of $E\in \CD$ is defined to be the cone of the canonical evaluation morphism:
\[
    \xymatrix{
   \Hom^{\bullet}(S,E)\otimes S \ar[r] & E \ar[r] & \Tw_S(E) \ar[r]^-{[1]} & 
    }
\]
\end{defi}
The following important lemma is due to Seidel and Thomas.
\begin{lem}[\cite{ST01}]
    Let $S$ be a spherical object in $\CD$. Then $\mr{Tw}_S$ is an exact autoequivalence of $\CD$.
\end{lem}
\begin{lem}
If $E$ is an exceptional object on $\pp$, then $s_* E$ is a 3-spherical object in $D^b(X)$.
\end{lem}
\begin{proof}
By using the Koszul resolution as in Proposition \ref{quiver_tilt}, we have
\begin{eqnarray*}
        \Ext^n_X(s_*E, s_* E) &\cong& \Ext^n_{\ff}(E,E) \oplus \Ext_{\ff}^{3-n}(E, E) \\
        &=& \left\{ \begin{array}{cc}
                            \mb{C}& n = 0,\ 3, \\
                            0 & \text{otherwise.}
                            \end{array}
                            \right.
\end{eqnarray*}
\end{proof}
Define $\tau: \pp \rightarrow \pp$, $\tau(x,y):= (y,x)$, it has a natural extension to an automorphism of $X = \Tot \co(-2,-2)$ which we also denote by $\tau$. We consider the following functor
\[
    \Psi:= \tau^* \circ \Tw_{S_0} \circ \left(-\otimes \pi^*\co(0,1)\right)
\]
which is an autoequivalence of $D^b(X)$ since it is a composition of autoequivalences.
\begin{lem}\label{aut_quiver}
Recall that $\ca$ denotes the heart of the bounded t-structure induced by $\mc{Q}$, and let $S_i$ be the simple objects in $\ca$ defined in Corollary \ref{simp_heart}. Then
    \[
        \Psi (S_i) = S_{i+1}, \quad i\in \mb{Z}_4.
    \]
    Therefore $\Psi$ reduces to be an autoequivalence of $\ca$.
\end{lem}
\begin{proof}
By the projection formula, we have
\begin{eqnarray*}
    s_*\co(a,b) \otimes \pi^*\co(j,k) &=& s_*\left(\co(a,b)\otimes s^*\pi^*\co(j,k)\right) \\
                                     &=& s_*\left(\co(a,b)\otimes \co(j,k)\right) \\
                                     &=& s_*\co(a+j,b+k).
\end{eqnarray*}
\begin{enumerate}
    \item Recall $S_3 = s_* \co(0,-1)[2]$. Thus $S_3 \otimes \pi^*\co(0,1) = s_*\co(0,0)[2] = S_0[2]$. Now
    \[
        \Psi(S_3) = \tau^* \Tw_{S_0} (S_0[2]) = \tau^* S_0 = \tau_* s_* \co(0,0) = S_0,
    \]
    where the second equality follows from the standard result, that if $S$ is an $n$-spherical object, then
    \[
        \Tw_S(S) \cong S[1-n],
    \]
    and $n=3$ in our case.
    \item For $S_0 = s_*\co(0,0)$, $S_0 \otimes \pi^*\co(0,1) = s_* \co(0,1)$. By using the similar calculations in Proposition \ref{quiver_tilt}, we have 
    $$
    \Hom^{\bullet}\left(s_*\co(0,0), s_*\co(0,1)\right) = \mb{C}^2,
    $$ 
    then $\Tw_{S_0}\left(s_*\co(0,1)\right) $ fits into the triangle:
    \[
        \xymatrix{
        s_*\co(0,0)^{\oplus 2} \ar[r] & s_*\co(0,1) \ar[r] & \Tw_{S_0}\left(s_*\co(0,1)\right) \ar[r]^-{[1]} & 
        }
    \]
   By applying the exact functor $s_*$ to the short exact sequence on $\ff$:
    \[
        \xymatrix{
            0 \ar[r] & \co(0,-1) \ar[r] & \co^{\oplus 2} \ar[r] & \co(0,1) \ar[r] & 0
        }
    \]
    and comparing with the above triangle, we have 
    \[
        \Tw_{S_0}\left(s_*\co(0,1)\right) \cong s_*\co(0,-1)[1].
    \]
    Thus $\Psi(S_0) = \tau^* s_*\co(0,-1)[1] = s_*\co(-1,0)[1] = S_1$.
    \item For $S_1 = s_*\co(-1,0)[1]$, $S_1 \otimes \pi^*\co(0,1) = s_*\co(-1,1)[1]$. Since $$\Hom^{\bullet}\left(s_*\co(0,0), s_*\co(-1,1)[1] \right) = 0,$$ then $\Tw_{S_0}\left(s_*\co(-1,1)[1] \right) \cong s_*\co(-1,1)[1]$. So we have 
    \[
        \Psi(S_1) = \tau^* s_*\co(-1,1)[1] = s_*\co(1,-1)[1] = S_2.
    \]
    \item For $S_2 = s_*\co(1,-1)[1]$, $S_2 \otimes \pi^*\co(0,1) = s_*\co(1,0)[1]$. Since 
    $$\Hom^{\bullet}\left(s_*\co(0,0), s_*\co(1,0)[1]\right) = \mb{C}^2[1],$$ then $\Tw_{S_0}\left(s_*\co(1,0)\right) $ fits into the triangle:
    \[
        \xymatrix{
        s_*\co(0,0)^{\oplus 2}[1] \ar[r] & s_*\co(1,0)[1] \ar[r] & \Tw_{S_0}\left(s_*\co(1,0)\right) \ar[r]^-{[1]} & 
        }
    \]
    by the same argument as above, we have     
    \[
        \Tw_{S_0}\left(s_*\co(1,0)\right) \cong s_*\co(-1,0)[2].
    \]
    So we have 
    \[
        \Psi(S_2) = \tau^* s_*\co(-1,0)[2] = s_*\co(0,-1)[2] = S_3.
    \]
\end{enumerate}
\end{proof}

\begin{defi}
We define the autoequivalence of $D^b_0(X)$
\begin{equation}
    \Phi = \Psi^2.
\end{equation}
 We denote by $\varphi$ and $\psi$ the automorphisms of $K_0(X)$ induced by $\Phi$ and $\Psi$ respectively.
\end{defi}
Let $\Stab(X)$ denote the space of stability conditions satisfying the support property on $D^b_0(X)$.
\begin{defi}[$\Phi$-invariant stability conditions]
   The space of  stability conditions which are invariant under $\Phi$ is denoted by $\Stab(X)^{\Phi}$.
     Let $\mc{U}(\ca)^{\Phi}$ be the set of $\Phi$-invariant stability conditions with the fixed heart $\ca$. We denote the connected component of $\Stab(X)^{\Phi}$ which contains $\mc{U}(\ca)^{\Phi}$ by 
        \[
            \bigl(\Stab(X)^{\Phi}\bigr)_0.
        \]
\end{defi}
From now on, we denote by $\gamma_i = [S_i]$ the class of $S_i$ in $K_0(\ca)$, $i=0,\cdots,3$.

The subgroup of $K_0(\ca)$ whose elements are antisymmetric under $\varphi$ is generated by $\gamma_0-\gamma_2$ and $\gamma_1-\gamma_3$, and is denoted by $K_0(\ca)^{-\varphi}$. The quotient group is denoted by 
\[
    \overline{K_0(\ca)} := K_0(\ca)/K_0(\ca)^{-\varphi}.
\]
The  quotient map is denoted by $\nu: K_0(\ca)\rightarrow \overline{K_0(\ca)}$. Note that $\overline{K_0(\ca)}$ is free abelian of rank $2$ with basis $\bar{\gamma_0}, \bar{\gamma_1}$ (we will abuse notation and still denote $\gamma_i$ in the quotient group). And there is a natural isomorphism
\[
    \Hom_{\mb{Z}}(\overline{K_0(\ca)},\mb{C}) \longrightarrow \Hom_{\mb{Z}}(K_0(\ca),\mb{C})^{\varphi}.
\]
Therefore we can equally define the $\Phi$-invariant stability conditions to be those whose central charges $Z: K_0(\ca)\rightarrow \mb{C}$ factor through $\overline{K_0(\ca)}$ and the slicings are invariant under $\Phi$. For technical reason we will work with this definition.

By Corollary \ref{cor:stab_mani_2}, the forgetful map $\bigl(\Stab(X)^{\Phi}\bigr)_0\rightarrow \Hom_{\mb{Z}}(\overline{K_0(\ca)},\mb{C})\cong \mb{C}^2 $ is a local homeomorphism. 

At the end of this section, we recall the following definition from \cite{BQS20}.
\begin{defi}\label{def:aut_group}
    Let $\Aut(D^b_0(X))$ be the group of exact $\mb{C}$-linear autoequivalences of the category $D^b_0(X)$, then $\Aut_*(D^b_0(X))$ is defined to be the subquotient consisting of autoequivalences which preserve the connected component $\bigl(\Stab(X)^{\Phi}\bigr)_0$, modulo those which acts trivially on it. 
\end{defi}
\section{Simple tilts and autoequivalence} \label{sec:double_tilts}
In this section we use Proposition \ref{sim_tilt} to figure out the double simple tilts of $\ca$, that is we will calculate $\ca' = L_{S_{i+2}}L_{S_i} \ca$ and $R_{S_{i+2}} R_{S_i} \ca$ for $i\in \BZ_4$.

We simply write $L_i := L_{S_i}$ and $R_i := R_{S_i}$.   Recall that there are $4$ simple objects in $\ca$ up to isomorphism:
\[
     S_0 = s_*\co(0,0), \quad S_1 = s_*\co(-1,0)[1], \quad S_2 = s_*\co(1,-1)[1], \quad S_3 = s_*\co(0,-1)[2].
\]
Proposition \ref{quiver_tilt} shows that there is no extension between $S_i$ and $S_{i+2}$, for $i\in \mb{Z}_4$, therefore $L_i L_{i+2}\ca = L_{i+2} L_i\ca$. 
\begin{enumerate}
    \item[$L_0 \ca$:] Since the only non-trivial extension to $S_0$ is $\Ext^1(S_1,S_0)$, thus the new simple objects are 
    \[
        S_0^{\prime} = S_0[1], \quad S_1^{\prime}, \quad S_2^{\prime} = S_2, \quad S_3^{\prime} = S_3,
    \]
    where $S_1^{\prime}$ fits into the  triangle
    \[
        S_0^{\oplus 2} \rightarrow S_1^{\prime} \rightarrow S_1 \rightarrow S_0^{\oplus 2}[1].
    \]
     Thus  $S_1'$ fits into the short exact sequence
    \begin{equation}\label{ses1}
        0 \rightarrow s_*\co(-1,0) \rightarrow s_*\co^{\oplus 2} \rightarrow  S_1' \rightarrow 0.
    \end{equation}
    We  already see the above short exact sequence in part 4 of the proof of Lemma \ref{aut_quiver}, therefore $S_1' = s_*\co(1,0)$.
    \item[$L_2 L_0 \ca$:] The new simple objects are 
    \[
        \wt{S_0} = S_0^{\prime}, \quad \wt{S_1} = S_1^{\prime}, \quad \wt{S_2} = S_2^{\prime}[1], \quad \wt{S_3}
    \]
    where $\wt{S_3}$ fits into the triangle
    \[
        S_2^{\oplus 2} \rightarrow \wt{S_3} \rightarrow S_3^{\prime} \rightarrow S_2^{\oplus 2}[1].
    \]
    Thus $\wt{S_3}[-1]$ fits into the short exact sequence
    \begin{equation}
        0 \rightarrow s_*\co(0,-1) \rightarrow s_*\co(1,-1)^{\oplus 2} \rightarrow \wt{S_3}[-1] \rightarrow 0.
    \end{equation}
    We obtain $\wt{S_3} = s_*\co(2,-1)[1]$.
\end{enumerate}
Therefore in $L_2L_0 \ca$ we have the following simple objects up  to isomorphism:
\[
    \wt{S_0} = s_*\co(0,0)[1],\quad \wt{S_1} = s_*\co(1,0),\quad \wt{S_2} = s_*\co(1,-1)[2], \quad \wt{S_3} = s_*\co(2,-1)[1].
\]
\begin{thm}\label{double_tilts}
 Let $\mc{T} = -\otimes \pi^*\co(1,0)$ and $\mc{T}_{\Psi} = \Psi\circ \mc{T} \circ \Psi^{-1}$, then we have 
 \begin{eqnarray*}
    L_2L_0\ca &=& \mc{T}\ca;\\
    R_3R_1\ca &=& \mc{T}^{-1}\ca;\\
    L_3L_1\ca &=& \mc{T}_{\Psi}\ca;\\
    R_2R_0 \ca&=& \mc{T}_{\Psi}^{-1} \ca.
 \end{eqnarray*}
\end{thm}
\begin{proof}
 $\mc{T} \ca\subset L_2L_0 \ca $ follows directly from the comparison of the simple objects after reordering them. By Lemma \ref{nest_t} we have $\mc{T} \ca = L_2L_0\ca$. 
 We also have 
    \begin{eqnarray*}
        \Psi\circ \mc{T}\circ \Psi^{-1}(\ca) 
        &=& \Psi L_2 L_0 \circ \Psi^{-1} \ca \\
        &=& \Psi \circ \Psi^{-1} L_3 L_1 \ca \quad \text{by Lemma \ref{aut_sim_tilt} and Lemma \ref{aut_quiver} }  \\
        &=& L_3 L_1 \ca.
    \end{eqnarray*}
    This proves the third identity. For the right mutation $R_3R_1\ca$, by using Remark \ref{inv_tilt} we have 
    $$
        R_{S_0[1]}R_{S_2[1]}L_{S_2}L_{S_0} \ca = \ca.
    $$
    Note that $R_{S_0[1]}R_{S_2[1]}L_{S_2}L_{S_0} \ca = R_{S_0[1]}R_{S_2[1]} \mc{T} \ca = \mc{T} R_{S_1}R_{S_3} \ca$ by Lemma \ref{aut_sim_tilt},  and $\mc{T}(S_1) = S_0[1]$,   $\mc{T}(S_3)= S_2[1]$. Therefore combining with the above identities we have 
    $$
        R_{S_1}R_{S_3} \ca = \mc{T}^{-1} \ca.
    $$
    Finally for $R_2 R_0 \ca$, the calculation is quite similar to that of $L_3L_1 \ca$ and we leave it to the reader. 
\end{proof}

$\mc{T}$ and $\mc{T}_{\Psi}$ induce automorphisms $t$ and $t_{\psi}$ of the Grothendieck group $K_0(\ca)$. The following results will be useful later:
\begin{lem}\label{auto_kgroup}
With respect to the  basis $\{\gamma_i\}$ of $K_0(\ca)$,  the automorphisms $t$ and $t_{\psi}$ have the matrix forms:
\begin{equation}
    t = \left[\begin{array}{c c c c}
                2 & -1 & 0 & 0 \\
                1 & 0 & 0 & 0 \\
                0 & 0 & 2 & -1 \\
                0 & 0 & 1 & 0
                \end{array}
            \right] \quad 
    t_{\psi} = \left[\begin{array}{c c c c}
                0 & 0 & 0 & 1 \\
                0 & 2 & -1 & 0 \\
                0 & 1 & 0 & 0 \\
                -1 & 0 & 0 & 2
                \end{array}
            \right] \quad 
\end{equation}  
\end{lem}
\begin{proof}
Follows directly from the calculations above.
 \end{proof}
It is easy to check that $t$ and $t_{\psi}$ preserve the subgroup $K_0(\ca)^{-\varphi}$. Therefore $t$ and $t_{\psi}$ can be regarded as the actions on $\overline{K_0(\ca)}$. In fact, when reducing to $\overline{K_0(\ca)}$, $t$ and $t_{\psi}$ have the matrix forms with respect to the basis $\{\gamma_0,\gamma_1\}$:
\begin{equation}
    t|_{\overline{K_0(\ca)}} = \left[\begin{array}{c c}
                2 & -1 \\
                1 & 0 
                \end{array}
            \right] \quad 
    t_{\psi}|_{\overline{K_0(\ca)}} = \left[\begin{array}{c c}
                0 & 1 \\
                 -1 & 2 
                \end{array}
            \right], 
\end{equation}  
we have 
 \begin{equation}\label{eqn:act_quot}
     \left(t|_{\overline{K_0(\ca)}}\right)^{-1} = t_{\psi}|_{\overline{K_0(\ca)}}.
 \end{equation}

We have the following relation in $K_0(\ca)$. 
\begin{lem}\label{ox_kgroup}
Let $x = ([a:b],[c:d])$ be a closed point in $\mb{P}^1\times \mb{P}^1$, $\delta$ be the class of $\co_x$ in $K_0(\ca)$. Then we have 
$$
    \delta = \sum_{i = 0}^3 \gamma_i.
$$
\end{lem}
\begin{proof}
We write $F_1 = s_*{\co}_{[a:b]\times \mb{P}^1}, \ F_2 = s_*\co_{[a:b]\times \mb{P}^1}(-1)[1]$. First we have the short exact sequences
\begin{eqnarray*}
     0 \rightarrow s_*\co(-1,0) \rightarrow s_* \co \rightarrow F_1 \rightarrow 0, \\
     0 \rightarrow s_*\co(0,-1) \rightarrow s_* \co(1,-1) \rightarrow F_2[-1] \rightarrow 0.
\end{eqnarray*}
This gives $[F_1] = \gamma_0+\gamma_1$, $[F_2] = \gamma_2 + \gamma_3$. Then we consider the short exact sequence
\begin{equation}\label{ses:delta}
    0 \rightarrow s_*\co_{[a:b]\times \mb{P}^1}(-1) \rightarrow s_*{\co}_{[a:b]\times \mb{P}^1} \rightarrow s_* \co_x \rightarrow 0
\end{equation}
which gives $\delta = [F_1] + [F_2] = \sum_i\gamma_i$ as required.
\end{proof}
The above lemma shows that $[\co_x]$ does not depend on $x\in \pp$.
\section{Semistable Objects} \label{sec:stab_obj}

\subsection{(Semi)stable objects}
 In this section we describe the set of stable objects for stability conditions $\sigma\in \mc{U}(\mc{A})^{\Phi}$ . The description relies on the known properties of stability conditions for the Kronecker quiver.

Denote by $K_2$ the Kronecker quiver 
$$
    \xymatrix{
     0 \ar@<0.5ex>[r] \ar@<-0.5ex>[r] & 1
    }
$$
and $\rep(K_2)$ the category of representations. We denote by $C_0$ and $C_1$ the simple objects at vertices $0$ and $1$. 
Recall the underlying quiver $Q$ of $\mc{A}$ (Proposition \ref{quiver_tilt}) is 
    \[
    \xymatrixcolsep{4pc}
    \xymatrix{
        0 \ar@<0.5ex>[r] \ar@<-0.5ex>[r] & 1 \ar@<0.5ex>[d] \ar@<-0.5ex>[d] \\
        3 \ar@<0.5ex>[u] \ar@<-0.5ex>[u] & 2 \ar@<0.5ex>[l] \ar@<-0.5ex>[l]
    }
    \]
\begin{defi}
    We define full subcategories  of $\ca \cong \rep_{\nil}(Q,I)$ which can naturally be identified with $\rep K_2$. The objects of the full subcategories are
    \begin{center}
    \begin{tabular}{|c |c| c|c|}
        \hline
        \textbf{Subcategories}   &  \textbf{Objects} & \textbf{Dimension vectors} &\textbf{Class in $K_0(\ca)$} \\
        \hline
        $\mc{K}^I_1$ & $\xymatrixcolsep{4pc}
                        \xymatrix{
                         \mb{C}^p \ar@<0.5ex>[r]^{\mu_1} \ar@<-0.5ex>[r]_{\mu_2} & \mb{C}^q \ar@<0.5ex>[d] \ar@<-0.5ex>[d] \\
                         0 \ar@<0.5ex>[u] \ar@<-0.5ex>[u] & 0 \ar@<0.5ex>[l] \ar@<-0.5ex>[l]
                         }$ & $(p,q,0,0)$ & $p\gamma_0 + q\gamma_1$\\
        \hline
        $\mc{K}^I_2$ & $\xymatrixcolsep{4pc}
                        \xymatrix{
                         0 \ar@<0.5ex>[r] \ar@<-0.5ex>[r] & 0 \ar@<0.5ex>[d] \ar@<-0.5ex>[d] \\
                         \mb{C}^q \ar@<0.5ex>[u] \ar@<-0.5ex>[u] & \mb{C}^p \ar@<0.5ex>[l]^{\mu_1} \ar@<-0.5ex>[l]_{\mu_2}
                         }$ & $(0,0,p,q)$ & $p\gamma_2+q\gamma_3$\\
        \hline
        $\mc{K}^{II}_1$ & $\xymatrixcolsep{4pc}
                        \xymatrix{
                         0 \ar@<0.5ex>[r] \ar@<-0.5ex>[r] & \mb{C}^p \ar@<0.5ex>[d]^{\mu_1} \ar@<-0.5ex>[d]_{\mu_2} \\
                         0 \ar@<0.5ex>[u] \ar@<-0.5ex>[u] & \mb{C}^q \ar@<0.5ex>[l] \ar@<-0.5ex>[l]
                         } $ & $(0,p,q,0)$ & $p\gamma_1+q\gamma_2$\\   
        \hline
        $\mc{K}^{II}_2$ & $\xymatrixcolsep{4pc}
                        \xymatrix{
                         \mb{C}^q \ar@<0.5ex>[r] \ar@<-0.5ex>[r] & 0 \ar@<0.5ex>[d] \ar@<-0.5ex>[d] \\
                         \mb{C}^p \ar@<0.5ex>[u]^{\mu_1} \ar@<-0.5ex>[u]_{\mu_2} & 0 \ar@<0.5ex>[l] \ar@<-0.5ex>[l]
                         }$  & $(q,0,0,p)$ & $p\gamma_3+q\gamma_0$ \\
        \hline
    \end{tabular}
     \end{center}
     The objects in $\mc{K}^I_i$, $i=1,\ 2$ are called Kronecker type I, and the objects in $\mc{K}^{II}_i$, $i=1,\ 2$ are called Kronecker type II.
\end{defi}
The following lemma is obvious.
\begin{lem}\label{embed_K2}
    The full subcategories $\mc{K}^{I}_i$ and $\mc{K}^{II}_i$,  $i=1,\ 2$ are equivalent to $\rep(K_2)$. They are Serre subcategories of $\ca \cong \rep_{\nil}(Q,I)$, i.e., they are closed under taking quotients and subobjects.
\end{lem}
We denote by $\Xi_{i}^{I}$ and $\Xi_i^{II}$ the corresponding embedding functors from the full subcategories to $\ca$.

Recall that for a finite acyclic quiver $Q$, $K_0(\rep Q)\cong \mb{Z}^{\oplus |Q_0|}$ is generated by the simple modules $S_i$ at each vertex $i$. We denote by $n_{ij}$ the number of arrows from vertex $i$ to $j$. Then the Euler form  on $K_0(\rep Q)$ is defined by $\chi([S_i], [S_j]) := \delta_{ij} - n_{ji}$.   For $\pmb{\alpha} = (\alpha_i)_{i\in Q_0}\in K_0(Q)$, the quadratic form  $q(-)$ is defined as $q(\pmb{\alpha}) := \chi(\pmb{\alpha},\pmb{\alpha})$. The associated matrix of $q$ is a symmetrization of the associated matrix of $\chi$. 

When $Q$ is Dynkin or affine Dynkin (for example, the Kronecker quiver), it is well-known that $q(-)$ is positive semi-definite. $\pmb{\alpha}$ is called a real root if $q(\pmb{\alpha}) = 1$ and an imaginary root if $q(\pmb{\alpha}) = 0$.  We need the following well-known result (for example, see \cite[Theorem 4.3.2]{BD98}).
\begin{thm}[Indecomposable representations of Kronecker quiver] \label{indec_K2}
We identify $K_0(\rep K_2)\cong \mb{Z}^2$ using the basis $([C_0],[C_1])$. Then
\begin{enumerate}
    \item for each real root $(n,n+1)$ or $(n+1,n)$ ($n\geq 0$), there is a unique indecomposable representation with this class in $K_0(\rep K_2)$, up to isomorphism, which we will denote by $E_{n,n+1}$ or $E_{n+1,n}$;
    \item for each imaginary root $(n,n)$ ($n \geq 1$), there is a family of indecomposable representations indexed by $\mb{P}^1$ with this class in $K_0(\rep K_2)$,  which we  denote by $E^{\lambda}_n$ where $\lambda = [a:b] \in \mb{P}^1$.
\end{enumerate}
The above are all the indecomposable representations in $\rep K_2$ up to isomorphism.
\end{thm}

We have the following characterization of stability conditions for Kronecker quivers by Okada \cite{Oka06}.
\begin{lem} \label{stab_K2}
Take a stability function $Z: K_0(\rep K_2) \rightarrow \mb{C}$ and denote by $\phi(E)$ the phase of a nonzero object $E\in \rep (K_2)$
\begin{enumerate}
    \item if $\phi(C_0) < \phi(C_1)$, then every indecomposable representation of $K_2$ is semistable, moreover, all indecomposable representations except for $E_m^{\lambda}$ when $m>1$ are stable.
    \item  If $\phi(C_1)<\phi(C_0)$, then the only stable objects are $C_0$ and $C_1$. The semistable objects are $C_0^{\oplus k},\ C_1^{\oplus k}$ for $k>1$.
    \item If $\phi(C_0) = \phi(C_1)$, then all objects are semistable, and only $C_0$,\ $C_1$ are stable.
\end{enumerate}
\end{lem}
\begin{proof}
Since $\rep K_2$ is of finite-length,  $Z$ satisfies the Harder-Narasimhan property automatically, therefore $Z$ can be extended to a stability condition for $D^b(K_2)\cong D^b(\mb{P}^1)$, and is denoted  by $(Z, \mc{P})$. 
\begin{enumerate}
    \item Let $T = \co \oplus \co(1)$ be the tilting object in $D^b(\mb{P}^1)$. Then the functor $\RHom(T,-)$ sends $\co$ and $\co(-1)[1]$ to $C_0$ and $C_1$ respectively. If $\phi(C_0) < \phi(C_1)$, after rotating by $\lambda=i\bigl(\pi-\phi(\co_x)\bigr)$ where $x$ is a closed point of $\mb{P}^1$, the resulting stability condition $\lambda \cdot (Z,\mc{P}) = (\overline{Z}, \overline{\mc{P}})$ has heart $\overline{\mc{P}} (0,1] = \mr{Coh}\ \mb{P}^1$ (see the following figures).
       \begin{figure}[!h]
\begin{center}
\begin{tikzpicture}
[thick,scale=0.8, every node/.style={scale=0.9}]
	\begin{pgfonlayer}{nodelayer}
		\node [style=none] (0) at (0, 0) {};
		\node [style=none] (1) at (4, 0) {};
		\node [style=none] (2) at (8, 0) {};
		\node [style=none] (3) at (1, 0) {};
		\node [style=none] (4) at (4, 2) {};
		\node [style=none] (5) at (2, 2) {};
		\node [style=none] (6) at (6, 2) {};
		\node [style=none] (7) at (8, 2) {};
		\node [style=none] (8) at (9, 2) {};
		\node [style=none] (9) at (9, 2) {$\dots$};
		\node [style=none] (10) at (1.25, 2.5) {$Z\bigl(\co(-1)\bigr)$};
		\node [style=none] (11) at (3.75, 2.5) {$Z(\co)$};
		\node [style=none] (12) at (6, 2.5) {$Z\bigl(\co(1)\bigr)$};
		\node [style=none] (13) at (8, 2.5) {$Z\bigl(\co(2)\bigr)$};
		\node [style=none] (14) at (1, 0.5) {$Z(\co_x)$};
		\node [style=none] (15) at (-6.5, 0) {};
		\node [style=none] (16) at (-10.5, 0) {};
		\node [style=none] (17) at (-2.5, 0) {};
		\node [style=none] (18) at (-4.5, 2.5) {};
		\node [style=none] (19) at (-8.75, 2) {};
		\node [style=none] (20) at (-6, 3) {};
		\node [style=none] (21) at (-3.5, 2) {};
		\node [style=none] (22) at (-2.5, 1.25) {};
		\node [style=none] (23) at (-2.25, 0.75) {$\ddots$};
		\node [style=none] (24) at (-9.5, 2.75) {$Z\bigl(\co(-1)[1]\bigr)$};
		\node [style=none] (25) at (-6, 3.75) {$Z(\co_x)$};
		\node [style=none] (26) at (-4.25, 3.25) {$Z(\co)$};
		\node [style=none] (27) at (-3, 2.5) {$Z\bigl(\co(1)\bigr)$};
		\node [style=none] (28) at (-2, 1.5) {$Z\bigl(\co(2)\bigr)$};
		\node [style=none] (29) at (-6.5, -0.75) {Central charges of $\rep K_2$};
		\node [style=none] (30) at (4, -0.75) {};
		\node [style=none] (31) at (4, -0.75) {Central charges after rotating};
		\node [style=none] (32) at (3.25, -1.75) {};
	\end{pgfonlayer}
	\begin{pgfonlayer}{edgelayer}
		\draw (0.center) to (1.center);
		\draw [shorten >=0pt, dashed](1.center) to (2.center);
		\draw [style={<-}](5.center) to (1.center);
            \draw [style={<-}](3.center) to (1.center);
		\draw [style={<-}](4.center) to (1.center);
		\draw [style={->}](1.center) to (6.center);
		\draw [style={->}](1.center) to (7.center);
		\draw (16.center) to (15.center);
		\draw [style={->}](15.center) to (19.center);
		\draw [style={->}](15.center) to (18.center);
		\draw [style={->}](15.center) to (20.center);
		\draw [shorten >=0pt, dashed](15.center) to (17.center);
		\draw [style={->}](15.center) to (21.center);
		\draw [style={->}](15.center) to (22.center);
	\end{pgfonlayer}
\end{tikzpicture}
\end{center}
\end{figure}
    Therefore all line bundles and torsion sheaves are semistable, and in fact all line bundles and skyscraper sheaves are stable. They correspond to the indecomposable representations of $K_2$ by the functor $\RHom(T,-)$.
    \item The second statement follows  from the fact that $C_0$ is a simple subobject of every indecomposable representation except $C_1$ and $C_1$ is a simple factor object of every indecomposable representation except $C_0$.
    \item If $\phi(C_0)=\phi(C_1)$, then all nonzero objects in $\Rep(K_2)$ have the same phase and therefore are semistable.
\end{enumerate}
\end{proof}

\begin{rem}
Since $\mc{K}^{I}_i$ and $\mc{K}^{II}_i$ are equivalent to $\rep(K_2)$ by Lemma \ref{embed_K2}, therefore by Theorem \ref{indec_K2} we can describe the indecomposable objects of Kronecker type I and II.   
\end{rem}
Since we will only be interested in the stable objects in $\rep(K_2)$ by Lemma \ref{stab_K2}, by definition the stable representations are bricks, that is, $\End(M) = \BC$ if $M\in \rep(K_2)$ is stable. Then the following definition will be useful:
\begin{defi}[Special Kronecker type]\label{def:spe_kron}
    We call the indecomposable object of Kronecker type I and II  {\bf special} if it is a brick, or equivalently it is not isomorphic to the image of $E^{\lambda}_m$ under $\Xi^I_i$ or $\Xi^{II}_i$ for $m>1$.
\end{defi}

The following proposition  gives us the geometric description of objects of special Kronecker types. The calculations are direct and we leave them to the reader.
\begin{prop}\label{kron_type}
Let $l\geq 0$. We have the following correspondences between objects in $\mc{A}$ and $\rep_{\mr{nil}}(Q,I)$ under the equivalence $\mc{R}_{\mc{Q}} :\mc{A} \rightarrow \rep_{\mr{nil}}(Q,I) $  (we denote by $x$ a closed point of $\mb{P}^1$):
\begin{center}
\begin{tabular}{|c|c|| c | c |}
    \hline
  \textbf{\small Objects  of special }   & \textbf{\small Classes in $K_0(\ca)$}   & \textbf{\small Objects of special} & \textbf{\small Classes in $K_0(\ca)$}\\
  \textbf{\small Kronecker type I} &   &\textbf{\small Kronecker type II}&\\
    \hline
    $s_*\mc{O}(l,0)$ & $(l+1)\gamma_0 + l \gamma_1$ &   $\Psi(s_*\mc{O}(l,0))$ & $(l+1)\gamma_1 + l \gamma_2$ \\
    \hline
    $s_*\mc{O}(l+1,-1)[1]$ & $(l+1)\gamma_2+l \gamma_3$ & $\Psi(s_*\mc{O}(l+1,-1)[1])$ & $(l+1)\gamma_3 + l \gamma_0$ \\
    \hline
    $s_*\mc{O}(-l-1,0)[1]$ & $l\gamma_0 + (l+1)\gamma_1$ & $\Psi(s_*\mc{O}(-l-1,0)[1])$ & $l\gamma_1 + (l+1)\gamma_2$\\
    \hline
    $s_*\mc{O}(-l,-1)[2]$ & $l\gamma_2+(l+1)\gamma_3$ & $\Psi(s_*\mc{O}(-l,-1)[2])$ & $l\gamma_3 + (l+1)\gamma_0$ \\
    \hline
    $F_1 = s_*\mc{O}_{\{x\}\times \mb{P}^1}$ & $\gamma_0+\gamma_1$ & $\Psi(s_*\mc{O}_{\{lx\}\times \mb{P}^1})$ & $\gamma_1+\gamma_2$ \\
    \hline
    $F_2 = s_*\mc{O}_{\{x\}\times \mb{P}^1}(-1)[1]$ & $\gamma_2+\gamma_3$ & $\Psi(s_*\mc{O}_{\{x\}\times \mb{P}^1}(-1)[1])$ & $\gamma_3+\gamma_0$\\
    \hline
\end{tabular}
\end{center}
\end{prop}

From now on, we often identify  objects in $\rep_{\mr{nil}}(Q,I)$ with  the corresponding objects in $\ca$ without further comment. 
\begin{defi}
    We introduce the open subsets of $\mc{U}(\ca)^{\Phi}$:
    \begin{eqnarray*}
    \mc{U}(\mc{A})^{\Phi}_+ &=& \{\sigma\in \mc{U}(\mc{A})^{\Phi}: \phi(S_0) = \phi(S_2) < \phi(S_1)= \phi(S_3)\}, \\
     \mc{U}(\mc{A})^{\Phi}_- &=& \{\sigma\in \mc{U}(\mc{A})^{\Phi}: \phi(S_1) = \phi(S_3) < \phi(S_0)= \phi(S_1)\}
    \end{eqnarray*}
    \begin{figure}[!h]
        \begin{center}
        \begin{minipage}{0.5\linewidth}
        \begin{tikzpicture}
	    \begin{pgfonlayer}{nodelayer}
		  \node [style=none] (0) at (-2.5, 0) {};
		  \node [style=none] (1) at (0, 0) {};
		  \node [style=none] (2) at (2.5, 0) {};
		  \node [style=none] (3) at (-2, 2.5) {};
		  \node [style=none] (4) at (2.5, 2) {};
		  \node [style=none] (7) at (2, 2.2) {$Z(\gamma_0)=Z(\gamma_2)$};
		  \node [style=none] (9) at (-2, 2.8) {$Z(\gamma_1)=Z(\gamma_3)$};
	   \end{pgfonlayer}
	   \begin{pgfonlayer}{edgelayer}
		  \draw (0.center) to (1.center);
		  \draw [shorten >=0pt,dashed] (1.center) to (2.center);
		  \draw [style={edge_arrow}] (3.center) to (1.center);
		  \draw [style={edge_arrow_2}] (1.center) to (4.center);
	   \end{pgfonlayer}
            \end{tikzpicture}
            \caption{Central charges of $\mc{U}(\ca)^{\Phi}_+$}
        \end{minipage}\hfill
        \begin{minipage}{0.5\linewidth}
        \begin{tikzpicture}
	    \begin{pgfonlayer}{nodelayer}
		  \node [style=none] (0) at (-2.5, 0) {};
		  \node [style=none] (1) at (0, 0) {};
		  \node [style=none] (2) at (2.5, 0) {};
		  \node [style=none] (3) at (-2, 2.5) {};
		  \node [style=none] (4) at (2.5, 2) {};
		  \node [style=none] (7) at (2, 2.2) {$Z(\gamma_1)=Z(\gamma_3)$};
		  \node [style=none] (9) at (-2, 2.8) {$Z(\gamma_0)=Z(\gamma_2)$};
	   \end{pgfonlayer}
	   \begin{pgfonlayer}{edgelayer}
		  \draw (0.center) to (1.center);
		  \draw [shorten >=0pt,dashed] (1.center) to (2.center);
		  \draw [style={edge_arrow}] (3.center) to (1.center);
		  \draw [style={edge_arrow_2}] (1.center) to (4.center);
	   \end{pgfonlayer}
            \end{tikzpicture}
            \caption{Central charges of $\mc{U}(\ca)^{\Phi}_-$}
        \end{minipage}
        \end{center}
    \end{figure}
\end{defi}
\begin{lem} \label{lem:po_ne}
The autoequivalence $\Psi$ induces a bijection between $\mc{U}(\mc{A})^{\Phi}_+ $ and $\mc{U}(\mc{A})^{\Phi}_-$.
\end{lem}
\begin{proof}
Given $\sigma = (Z,\ca)\in \mc{U}(\mc{A})^{\Phi}_+$, we denote by $\Psi(\sigma) = (Z_{\psi}, \Psi(\ca)=\ca)$. By Lemma \ref{aut_quiver}, $Z_{\psi}(\gamma_i) = Z(\gamma_{i-1})$, therefore $\Psi(\sigma)\in \mc{U}(\mc{A})^{\Phi}_-$. The statement follows immediately.
\end{proof}
Let $\sigma = (Z,\ca) \in \mc{U}(\ca)^{\Phi}$. We define a stability function $\wbar{Z}$ on $\rep K_2$ by setting $\wbar{Z}^{I}:= Z\circ \Xi^I_i$ and $\wbar{Z}^{II} = Z\circ \Xi^{II}_i$, where $\Xi^I_{i}$ $i=1,\ 2$ are the embedding functors.
\begin{lem}\label{lem:restr_stab}
    The stable objects in $\mc{K}^I_i$ (resp. $\mc{K}^{II}_i$) with respect to $\wbar{Z}^I$ (resp. $\wbar{Z}^{II}$) are stable in $\ca$ with respect to $\sigma$.
\end{lem}
\begin{proof}
By Lemma \ref{embed_K2} $\rep K_2$ is closed under taking quotients and subobjects, and since $\Xi^{I}_i$ and $\Xi^{II}_i$ preserve the ordering by phases, it follows that if $E\in \rep K_2$ is stable, then $\Xi^{I}_i(E)$ and $\Xi^{II}_i(E)$ are stable in $\ca$.
\end{proof}
We immediately have the following theorem
\begin{thm}\label{thm:stab_kron}
    Let $\sigma\in \mc{U}(\ca)^{\Phi}_+$. Then the objects of special Kronecker type I are stable for $\sigma$. For $\tau\in \mc{U}(\ca)^{\Phi}_-$, then the objects of special Kronecker type II are stable for $\tau$.
\end{thm}

\begin{cor} \label{ox_ss}
    Let $x\in \mb{P}^1\times \mb{P}^1$, then $\co_x$ is semistable with respect to the stability condition $\sigma\in \mc{U}(\ca)^{\Phi}_+$.  For $\sigma \in \mc{U}(\ca)^{\Phi}_-$, there is a  semistable object whose class in $K_0(\ca)$ is $\delta = [\co_x]$.
\end{cor}
\begin{proof}
Let  $F_1 = s_*\mc{O}_{\{x_1\}\times \mb{P}^1}, \ F_2 = s_*\mc{O}_{\{x_1\}\times \mb{P}^1}(-1)[1]$ where $x_1 = p_1(x)$. By (\ref{ses:delta}), there is a short exact sequence in $\ca$:
$$
    0 \rightarrow F_1 \rightarrow \co_x \rightarrow F_2 \rightarrow 0.
$$
By Proposition \ref{kron_type} and Theorem \ref{thm:stab_kron} $F_1$ and $F_2$ are stable for the stability condition $\sigma\in\mc{U}(\ca)^{\Phi}_+$. Moreover, since $Z(F_1) = Z(\gamma_0)+Z(\gamma_1)$ and $Z(F_2)=Z(\gamma_2)+Z(\gamma_3)$ therefore $\phi(F_1) = \phi(F_2)$.  So  $\co_x$ is (strictly) semistable for $\sigma$ with the same phase as $\phi(F_i)$.

Suppose $\sigma\in \mc{U}(\ca)^{\Phi}_-$. We take $\Psi^{-1}(\sigma)$, by Lemma \ref{lem:po_ne} $\Psi^{-1}(\sigma) \in \mc{U}(\ca)^{\Phi}_+$. Therefore $\Psi(\co_x)$ is semistable for $\sigma$.  Since 
    \[
    \psi(\delta) = \sum_{i=0}^3 \psi(\gamma_i) = \sum_{i=0}^3 \gamma_i = \delta,
    \]
    the claim is proved.   
\end{proof}
\begin{rem}\label{rem:stab}
    For $\sigma\in \mc{U}(\ca)^{\Phi}$, we can conclude that there is a semistable object whose class in $K_0(\ca)$ is $\delta$. By the above corollary, the  remaining case we need to verify is that when $\phi(S_i) = \phi(\co_x)$ for all $i$, however, each object in $\ca$ is   semistable in this case.
\end{rem}
The central charges of $\delta$ and other stable objects for $\sigma\in \mc{U}(\ca)^{\Phi}_+$ are depicted in the figure \ref{fig:sstab}.
\begin{figure}[H]
    \centering
    \begin{tikzpicture}
	\begin{pgfonlayer}{nodelayer}
		\node [style=none] (0) at (-4, 0) {};
		\node [style=none] (1) at (0, 0) {};
		\node [style=none] (2) at (4.25, 0) {};
		\node [style=none] (3) at (-3, 0.75) {};
        \node [style=none] (17) at (-3, 3) {$\vdots$};
		\node [style=none] (4) at (3, 0.75) {};
		\node [style=none] (7) at (4.5, 0.75) {\small $Z(\gamma_0)=Z(\gamma_2)$};
		\node [style=none] (9) at (-4.25, 0.75) {\small $Z(\gamma_1)=Z(\gamma_2)$};
		\node [style=none] (10) at (-3, 2) {};
		\node [style=none] (11) at (-3, 4) {};
		\node [style=none] (16) at (3, 4) {};
		\node [style=none] (17) at (3, 2) {};
        \node [style=none] (26) at (3, 3) {$\vdots$};
		\node [style=none] (18) at (0, 3) {};
        \node [style=none] (26) at (0, 2) {};
        \node [style=none] (27) at (0.2, 2) {\small $Z(\gamma_0+\gamma_1) = Z(\gamma_2+\gamma_3)$};
		\node [style=none] (19) at (-3.75, 4.25) {};
		\node [style=none] (20) at (-3.75, 4.25) {\small $Z(n\gamma_0+(n+1)\gamma_1) = Z(n\gamma_2+(n+1)\gamma_3)$};
		\node [style=none] (21) at (3.75, 4.25) {};
		\node [style=none] (22) at (3.75, 4.25) {\small $Z((n+1)\gamma_0+n\gamma_1) = Z((n+1)\gamma_2+n\gamma_3)$};
		\node [style=none] (23) at (3.75, 4.25) {};
		\node [style=none] (24) at (0, 3.5) {};
		\node [style=none] (25) at (0, 3.5) {\small $Z(\delta)$};
	\end{pgfonlayer}
	\begin{pgfonlayer}{edgelayer}
		\draw (0.center) to (1.center);
		\draw [shorten >=0pt,dashed] (1.center) to (2.center);
		\draw [style={<-}] (3.center) to (1.center);
		\draw [style={->}, in=195, out=15] (1.center) to (4.center);
		\draw [style={<-}] (10.center) to (1.center);
		\draw [style={<-}] (11.center) to (1.center);
		\draw [style={<-}] (16.center) to (1.center);
		\draw [style={<-}] (17.center) to (1.center);
		\draw [style={<-}] (18.center) to (1.center);
        \draw [style={<-}] (26.center) to (1.center);
	\end{pgfonlayer}
\end{tikzpicture}
    \caption{Central charges of stable objects and $\co_x$ for $\sigma \in \mc{U}(\ca)^{\Phi}_+$}
    \label{fig:sstab}
\end{figure}

 Recall that $\mc{T}= -\otimes \pi^*\co(1,0)$ and $\mc{T}_{\Psi}=\Psi\circ\mc{T} \circ \Psi^{-1}$, the simple objects in $\ca$ are $S_0 = s_*\co(0,0)$, $S_1 = s_*\co(-1,0)[1]$, $s_*\co(1,-1)[1]$ and $S_3 = s_*\co(0,-1)[2]$. Finally we mention another description of some objects of special Kronecker types I and II,
\begin{lem}\label{lem:kron_type}
Let $n\geq 0$
    \begin{enumerate}
        \item $\mc{T}^n(S_0),\ \mc{T}^n(S_2),\ \mc{T}^{-n}(S_1),\ \mc{T}^{-n}(S_3)$ are objects of special Kronecker type I with classes in $K_0(\ca)$:
        \[ (n+1)\gamma_0 + n \gamma_1, \quad (n+1)\gamma_2+n\gamma_3,\quad  n\gamma_0+ (n+1)\gamma_1, \quad n\gamma_2 + (n+1)\gamma_3.
        \]
            
        \item $\mc{T}^n_{\Psi}(S_1),\ \mc{T}^n_{\Psi}(S_3),\ \mc{T}^{-n}_{\Psi}(S_0),\ \mc{T}^{-n}_{\Psi}(S_2)$ are objects of special Kronecker type II with classes in $K_0(\ca)$:
        \[
        (n+1)\gamma_1+ n\gamma_2,\quad (n+1)\gamma_3+n\gamma_0,\quad  n\gamma_1+(n+1)\gamma_2,\quad n\gamma_3 + (n+1)\gamma_0.
        \]
    \end{enumerate}
\end{lem}
\begin{proof}
    Note that $\mc{T}^n\bigl(s_*\co(a,b)\bigr) =  s_*\co(a+n,b)$ by the projection formula, the result follows directly from the table in Proposition \ref{kron_type}.
\end{proof}
For $\sigma\in \mc{U}(\ca)^{\Phi}_+$, we can alternatively illustrate the central charges of semistable objects in the complex plane:
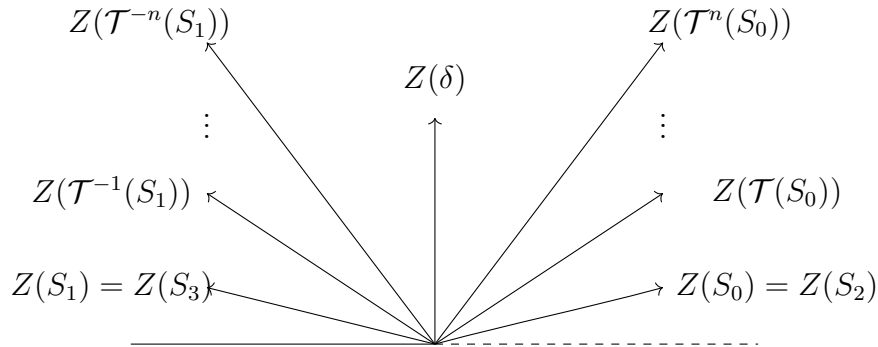
\begin{center}
    \begin{tikzpicture}
	\begin{pgfonlayer}{nodelayer}
		\node [style=none] (0) at (-4, 0) {};
		\node [style=none] (1) at (0, 0) {};
		\node [style=none] (2) at (4.25, 0) {};
		\node [style=none] (3) at (-3, 0.75) {};
        \node [style=none] (17) at (-3, 3) {$\vdots$};
		\node [style=none] (4) at (3, 0.75) {};
		\node [style=none] (7) at (4.5, 0.75) {$Z(S_0)=Z(S_2)$};
		\node [style=none] (9) at (-4.25, 0.75) {$Z(S_1)=Z(S_3)$};
		\node [style=none] (10) at (-3, 2) {};
		\node [style=none] (11) at (-3, 4) {};
		\node [style=none] (16) at (3, 4) {};
		\node [style=none] (17) at (3, 2) {};
        \node [style=none] (26) at (3, 3) {$\vdots$};
		\node [style=none] (18) at (0, 3) {};
		\node [style=none] (19) at (-3.75, 4.25) {};
		\node [style=none] (20) at (-3.75, 4.25) {$Z(\mc{T}^{-n}(S_1))$};
		\node [style=none] (21) at (3.75, 4.25) {};
		\node [style=none] (22) at (3.75, 4.25) {$Z(\mc{T}^n(S_0))$};
		\node [style=none] (23) at (3.75, 4.25) {};
		\node [style=none] (24) at (0, 3.5) {};
		\node [style=none] (25) at (0, 3.5) {$Z(\delta)$};
        \node [style=none] (26) at (-4.25, 2) {$Z(\mc{T}^{-1}(S_1))$};
        \node [style=none] (27) at (4.5, 2) {$Z(\mc{T}(S_0))$};
	\end{pgfonlayer}
	\begin{pgfonlayer}{edgelayer}
		\draw (0.center) to (1.center);
		\draw [shorten >=0pt,dashed] (1.center) to (2.center);
		\draw [style={<-}] (3.center) to (1.center);
		\draw [style={->}, in=195, out=15] (1.center) to (4.center);
		\draw [style={<-}] (10.center) to (1.center);
		\draw [style={<-}] (11.center) to (1.center);
		\draw [style={<-}] (16.center) to (1.center);
		\draw [style={<-}] (17.center) to (1.center);
		\draw [style={<-}] (18.center) to (1.center);
	\end{pgfonlayer}
\end{tikzpicture}
\end{center}
\subsection{There are no other stable objects}\label{sec:noother_stab}
This subsection is the main part of this paper. We prove that for $\sigma \in \mc{U}(\ca)^{\Phi}$, there are no other stable objects other than the ones in Theorem \ref{thm:stab_kron}. 

For simplicity,  we first restrict ourselves to the normalized stability conditions 
\[
\bigl(\Stab(X)^{\Phi}\bigr)_n := \{\sigma=(Z,\mc{P}): Z(\delta) = i\}\subset \bigl(\Stab(X)^{\Phi}\bigr)_0,
\]
where $\delta$ is the class of skyscrapper sheaf $\co_x$ in $K_0(\ca)$. Note that $\bigl(\Stab(X)^{\Phi}\bigr)_n$ is a connected submanifold of $\bigl(\Stab(X)^{\Phi}\bigr)_0$. 

Let $\mc{U}^n(\mc{A})^{\Phi} \subset \bigl(\Stab(X)^{\Phi}\bigr)_n$ which consists of normalized stability conditions $(Z, \ca)$ with the fixed heart $\ca$, and $\mc{U}^n(\ca)^{\Phi}_{\pm} = \{(Z,\ca)\in \mc{U}(\ca)^{\Phi}_{\pm}: Z(\delta) = i \}$. For $\sigma\in \mc{U}^n(\mc{A})^{\Phi}_+$, we have \begin{equation}
        \phi(S_0)=\phi(S_2) < \phi(\co_x) = \frac{1}{2} < \phi(S_1) = \phi(S_3).
\end{equation}

Recall  the group action on the stability conditions in  Definition \ref{defi:group_act}. Let  $\mc{T}$, $\mc{T}_{\Psi}$ be the autoequivalences defined in Theorem \ref{double_tilts}. We first make  the following important observartion
\begin{lem} \label{g_act}
\begin{enumerate}
    \item For $\sigma\in \mc{U}^n(\mc{A})^{\Phi}_+ $, then $\mc{T}(\sigma) = \sigma\cdot \widetilde{g}$ where  $\widetilde{g} = (g, f)\in \widetilde{\mr{GL}}^+(2, \mb{R})$ such that $f: \mb{R}\rightarrow \mb{R}$ satisfies  $f\left(\frac{1}{2}\right) = \frac{1}{2}$.
    \item For $\sigma\in \mc{U}^n(\mc{A})^{\Phi}_- $, then $\mc{T}_{\Psi}(\sigma) = \sigma\cdot \widetilde{g}$ where  $\widetilde{g} = (g, f)\in \widetilde{\mr{GL}}^+(2, \mb{R})$ such that $f: \mb{R}\rightarrow \mb{R}$ satisfies  $f\left(\frac{1}{2}\right) = \frac{1}{2}$.
\end{enumerate}

\end{lem}
\begin{proof}
Let $\sigma \in\mc{U}^n(\mc{A})^{\Phi}_+$. We write $\mc{T}(\sigma)  = (Z_t, \mc{P}_{\mc{T}})$. By viewing $\mb{C} \cong \mb{R}^2$, we let 
$$
\pmb{e_0} = Z(\gamma_0) = Z(\gamma_2), \quad \pmb{e_1} = Z(\gamma_1) = Z(\gamma_3).
$$
 Then by (\ref{eqn:act_quot}), we have 
\begin{eqnarray*}
    Z_t(\gamma_0) &=& Z\left(t^{-1}\gamma_0\right) = Z\left(-\gamma_1\right) = -\pmb{e_1};\\
    Z_t(\gamma_1) &=& Z\left(t^{-1}\gamma_1\right) = Z\left(\gamma_0+2 \gamma_1\right) = \pmb{e_0}+2\pmb{e_1}.
\end{eqnarray*}
We define $g\in \mr{GL}^+(2,\mb{R})$ such that 
$$
    g(\pmb{e_0}) = 2\pmb{e_0}+\pmb{e_1}, \quad g(\pmb{e_1}) = -\pmb{e_0},
$$
Note that $2(\pmb{e_0}+\pmb{e_1}) = \sum_i Z(\gamma_i) = Z(\delta) = i$ by Lemma \ref{ox_kgroup}, and $g(\pmb{e_0}+\pmb{e_1}) = \pmb{e_0}+\pmb{e_1}$, therefore we see $g$ preserves the positive imaginary axis. We can take $\widetilde{g} = (g,f) \in \widetilde{\mr{GL}}^+(2, \mb{R})$ be the unique lift of $g$ such that $f(\frac{1}{2}) = \frac{1}{2}$. Let $\sigma\cdot \widetilde{g} = (Z_g, \mc{P}_f)$. Then by definition $Z_g(\gamma_i) = Z_t(\gamma_i)$ for $i=0,\ 1$. Therefore $Z_g = Z_t$.

We claim that the bounded hearts  $\mc{P}_{\mc{T}}(0,1]$ and $\mc{P}_f(0,1]$ are the same, thus finishing the proof of  the first case. By Theorem \ref{double_tilts} we have 
$$
    \mc{P}_{\mc{T}}(0,1] = \mc{T}(\ca) = L_{S_0} L_{S_2} \ca.
$$
Suppose $S_1, S_3 \in \mc{P}(\phi)$, $\phi\in (0,1]$. Since $g(\pmb{e_1}) = -\pmb{e_0} = -Z(\gamma_0) = -Z(\gamma_2)$, therefore $f(\phi) = \phi(S_0[i]) = \phi(S_2[i])$ for some odd number $i$. Note that $\phi \in (1/2,1] \subset (1/2,3/2)$, so
$$
    f(\phi) \in \left(f(\frac{1}{2}),f(\frac{3}{2})\right) = (1/2,3/2).
$$ 
By our assumption, \[\phi(S_0[1]) = \phi(S_2[1]) \in (1/2,3/2).\]
Therefore, $i=1$ and $S_0[1], S_2[1] \in \mc{P}_f(\phi)$ which is contained in $\ca' = \mc{P}_f(0,1]$. 

Suppose $S_0, S_2 \in \mc{P}(\omega)$, $\omega \in (0,1]$. Then $g(\pmb{e_0}) = 2\pmb{e_0}+\pmb{e_1} = Z(\widetilde{S_1}) = Z(\widetilde{S_3})$ where $\widetilde{S_1} = \mc{T}(S_0)$ and $\widetilde{S_3} = \mc{T}(S_2)$. By Lemma \ref{lem:kron_type} they are of special Kronecker type I, therefore they are semistable for $\sigma$ by Theorem \ref{thm:stab_kron}. We have $f(\omega) = \phi(\widetilde{S_1}[n]) = \phi(\widetilde{S_3}[n])$ for some even number $n$.  Since $\omega\in (0,1/2) \subset (-1/2,1/2)$, so $f(\omega) \in (-1/2,1/2)$. By our assumption, $\phi(\widetilde{S_1}) = \phi(\widetilde{S_2}) \in (-1/2,1/2)$. Therefore $n=0$ and $\widetilde{S_1}, \widetilde{S_3}\in \ca'$.

Recall the simple objects in $L_{S_0}L_{S_1}\ca$ are exactly $S_0[1]$, $S_2[1]$, $\wtil{S_1}$ and $\wtil{S_3}$ by our computations in Section \ref{sec:double_tilts}. We just proved $L_{S_0}L_{S_2} \ca =\langle S_0[1], S_2[1], \widetilde{S_1}, \widetilde{S_3}\rangle \subset \ca'$,  so by Lemma \ref{nest_t} the two hearts are equivalent. We finished the proof of the first case.  

For $\sigma \in \mc{U}^n(\mc{A})^{\Phi}_-$, note that $\Psi^{-1}(\sigma)\in \mc{U}^n(\mc{A})^{\Phi}_+$ by Lemma \ref{lem:po_ne}, therefore by the first part we have
\begin{eqnarray*}
    \mc{T}_{\Psi}(\sigma) &=& \Psi\circ\mc{T}\circ \Psi^{-1}(\sigma) \\
    &=&\Psi(\Psi^{-1}(\sigma)\cdot \wtil{g})\\
    &=& \sigma\cdot \wtil{g},
\end{eqnarray*}
where $\wtil{g} = (g,f)\in\wtil{\GL}^+(2,\mb{R})$ such that $f(1/2) = 1/2$. This finishes the proof.
\end{proof}
\begin{cor}\label{inv_t}
If $\sigma = (Z, \mc{P})\in \mc{U}^n(\mc{A})^{\Phi}_+$, then $\mc{T}^{\pm 1}\left(\mc{P}(\frac{1}{2})\right)= \mc{P}(\frac{1}{2})$. Similarly, if $\tau = (W, \mc{P}')\in \mc{U}^n(\mc{A})^{\Phi}_-$, then $\mc{T}_{\Psi}^{\pm 1}\left(\mc{P}'(\frac{1}{2})\right) = \mc{P}'(\frac{1}{2})$.
\end{cor}
\begin{proof}
Let $\sigma = (Z, \mc{P}) \in \mc{U}^n(\mc{A})^{\Phi}_+$ and $\mc{T}(\sigma) = (Z_t, \mc{P}_{\mc{T}})$.  By the above Lemma, $\mc{T}(\sigma) = \sigma\cdot \widetilde{g} = (Z_g, \mc{P}_f)$ for $\widetilde{g} = (g,f)\in \widetilde{\mr{GL}}^+(2,\mb{R})$ such that $f(1/2) = 1/2$ , we have
$$
    \mc{P}_f(1/2) = \mc{P}(1/2).
$$ 
 The proof for the second statement is the same.
\end{proof}
Recall for any interval $I\subset \mb{R}$, $\mc{P}(I)$ is the extension-closed subcategory of $D^b_0(X)$ generated by the subcategories $\mc{P}(\phi)$ for $\phi\in I$. Recall the definition of $\Aut_*(D^b_0(X))$ in Definition \ref{def:aut_group}. The following  proposition will be useful:
\begin{prop}\label{prop:act_order}
Let  $W$ be an element of $\Aut_*(D^b_0(X))$ such that for a stability condition $\sigma = (Z,\mc{P})$ we have $W(\sigma) = \sigma\cdot \wtil{g}$ for some $\wtil{g} = (g,f)\in\wtil{\GL}^+(2,\mb{R})$. Suppose that $E,\ F$ are two semistable objects with phases $\phi_E<\phi_F$, then 
\[
W\bigl(\mc{P}(\phi_E,\phi_F)\bigr) = \mc{P}\bigl(\phi_{W(E)},\phi_{W(F)}\bigr).
\]
\end{prop}
\begin{proof}
   We write $\sigma\cdot \wtil{g} = (Z_g,\mc{P}_f)$, then  $\mc{P}_f(I) = \mc{P}\bigl(f(I)\bigr)$ for any interval $I\subset \mb{R}$ by definition. Therefore by our assumption
    \[
        W\bigl(\mc{P}(\phi_E,\phi_F)\bigr) = \mc{P}\bigl(f(\phi_E),f(\phi_F)\bigr).
    \]    
    Since $\sigma\cdot \wtil{g}$ and $\sigma$ contain the same set of semistable objects, therefore $W(E)$ and $W(F)$ are semistable for $\sigma$. So we have $\phi_{W(E)} = f\bigl(\phi_E\bigr)$ and $\phi_{W(F)} = f\bigl(\phi_F\bigr)$, this proves the result.
\end{proof}
\begin{lem} \label{lem:cone}
Given a stability condition $\sigma = (Z, \ca)$ such that $\ca$ is of finite-length with the finite set of simple objects (up to isomorphism) $\{S_0, S_1, \cdots, S_n\}$. Define the linear cone in $\mb{C}$:
$$
    C:= \{z\in \mb{C}: z=\sum_{i=0}^n \lambda_i Z(S_i), \ \lambda_i \in \mb{Z}_{\geq 0} \} \setminus \{0\}.
$$
Then for any non-zero semistable object $E\in \mc{P}(\psi)$, $\psi \in \mb{R}$, we have 
$$
    Z(E) \in C \cup (-C).
$$
\end{lem}
\begin{proof}
After shifting $E$, we may assume that $E\in \ca$. Since $\ca$ is of finite-length,  $E$ has a finite filtration by the simple objects $S_i$. Then in $K_0(\ca)$ we have $[E] = \sum_{i=0}^n \lambda_i [S_i], \ \lambda_i\geq 0$ (at least one $\lambda_j \neq 0$). Therefore the result follows from the linearity of $Z$.
\end{proof}
The following important theorem characterizes the stable objects outside the ray $\phi = \frac{1}{2}$ in the upper half complex plane:
\begin{thm}
 Take $\sigma\in\mc{U}^n(\mc{A})^{\Phi}_+$,  there is no stable object whose phase lies in the intervals $\bigl(0,\phi(S_0)\bigr)$, $ \bigl(\phi(S_1),1\bigr)$, nor in the intervals
    \[
\Bigl(\phi\left(\mc{T}^m(S_0)\right),\phi\left(\mc{T}^{m+1}(S_0)\right) \Bigr) \text{ or } \Bigl(\phi\left(\mc{T}^{-m-1}(S_1)\right), \phi\left(\mc{T}^{-m}(S_1)\right)\Bigr),
    \]
    for any integer $m \geq 0$. Moreover, the stable objects of phases $\phi\bigl(\mc{T}^m(S_0)\bigr)$ and $\phi\bigl(\mc{T}^{-m}(S_1)\bigr)$ for $m\geq 0$ are of special Kronecker type I.
\end{thm}

\begin{proof}
Given $\sigma =(Z,\mc{P}) \in \mc{U}^n(\ca)^{\Phi}_+$, it is clear that there is no stable object of phase in the interval $(0,\phi(S_0))\cup (\phi(S_1),1)$ by Lemma \ref{lem:cone}. 

 By Lemma \ref{g_act} $\mc{T}(\sigma) = \sigma\cdot \wtil{g}$ for some $\wtil{g}\in \wtil{\GL}^+(2,\mb{R})$, therefore by Proposition \ref{prop:act_order} we only need to check that there is no stable object of phase in the intervals $\left(\phi(S_0), \phi(\mc{T}(S_0))\right)$ and $\left( \phi(\mc{T}^{-1}(S_1)), \phi(S_1)\right)$, then apply $\mc{T}^{\pm m}$ we see that there is no stable object of phase in other open intervals.
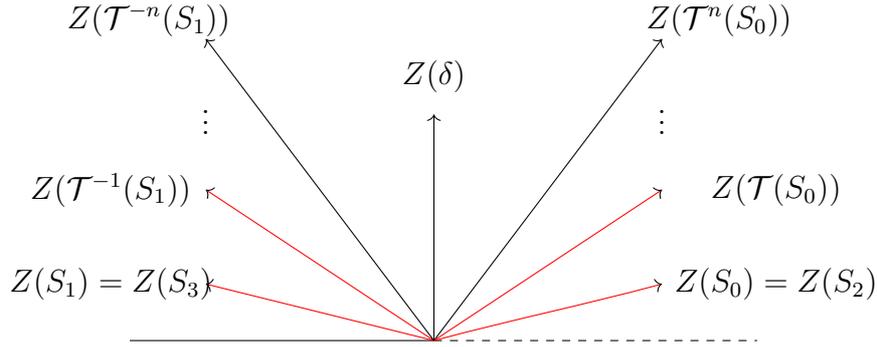
\begin{center}
    \begin{tikzpicture}
	\begin{pgfonlayer}{nodelayer}
		\node [style=none] (0) at (-4, 0) {};
		\node [style=none] (1) at (0, 0) {};
		\node [style=none] (2) at (4.25, 0) {};
		\node [style=none] (3) at (-3, 0.75) {};
        \node [style=none] (17) at (-3, 3) {$\vdots$};
		\node [style=none] (4) at (3, 0.75) {};
		\node [style=none] (7) at (4.5, 0.75) {$Z(S_0)=Z(S_2)$};
		\node [style=none] (9) at (-4.25, 0.75) {$Z(S_1)=Z(S_3)$};
		\node [style=none] (10) at (-3, 2) {};
		\node [style=none] (11) at (-3, 4) {};
		\node [style=none] (16) at (3, 4) {};
		\node [style=none] (17) at (3, 2) {};
        \node [style=none] (26) at (3, 3) {$\vdots$};
		\node [style=none] (18) at (0, 3) {};
		\node [style=none] (19) at (-3.75, 4.25) {};
		\node [style=none] (20) at (-3.75, 4.25) {$Z(\mc{T}^{-n}(S_1))$};
		\node [style=none] (21) at (3.75, 4.25) {};
		\node [style=none] (22) at (3.75, 4.25) {$Z(\mc{T}^n(S_0))$};
		\node [style=none] (23) at (3.75, 4.25) {};
		\node [style=none] (24) at (0, 3.5) {};
		\node [style=none] (25) at (0, 3.5) {$Z(\delta)$};
        \node [style=none] (26) at (-4.25, 2) {$Z(\mc{T}^{-1}(S_1))$};
        \node [style=none] (27) at (4.5, 2) {$Z(\mc{T}(S_0))$};
	\end{pgfonlayer}
	\begin{pgfonlayer}{edgelayer}
		\draw (0.center) to (1.center);
		\draw [shorten >=0pt,dashed] (1.center) to (2.center);
		\draw [style={<-, draw=red}] (3.center) to (1.center);
		\draw [style={draw=red, ->}, in=195, out=15] (1.center) to (4.center);
		\draw [style={draw=red,<-}] (10.center) to (1.center);
		\draw [style={<-}] (11.center) to (1.center);
		\draw [style={<-}] (16.center) to (1.center);
		\draw [style={draw=red,<-}] (17.center) to (1.center);
		\draw [style={<-}] (18.center) to (1.center);
	\end{pgfonlayer}
\end{tikzpicture}
\end{center}
Suppose $E\in \mc{P}\left(\phi(S_0)+\epsilon\right)$ for $0< \epsilon< \phi(\mc{T}(S_0))-\phi(S_0)$. We will take a $\mb{C}$-action on $\sigma$ and reduce to the case in the beginning:  we choose $0< \epsilon' \ll \epsilon$ and let $\lambda = -(\phi(S_0)+\epsilon')$. Then let $\sigma' := \sigma\cdot \lambda = (Z', \mc{P}')$. The phase of the semistable object for $\sigma'$ is denoted by $\phi'(-)$. By definition of the $\mb{C}$-action (see Remark \ref{rem:C_act}) we have 
\begin{eqnarray*}
    \phi'(S_0[1]) = \phi'(S_2[1]) &=& 
                \phi(S_0[1]) - \phi(S_0)-\epsilon'\\
                &=& 1-\epsilon'\in (1/2,1), \\
    \phi'\left(\mc{T}(S_0)\right) = \phi'\left(\mc{T}(S_2)\right) &=& \phi(\mc{T}(S_0)) - \phi(S_0)-\epsilon' \in (0,1/2).
\end{eqnarray*}
We see that the heart $\ca' = \mc{P}'(0,1]$ contains the simple objects $\{S_0[1],S_2[1],\mc{T}(S_0),\mc{T}(S_2)\}$ which generate $L_{S_0}L_{S_2}\ca$, we have $L_{S_0}L_{S_2}(\ca)\subset \ca'$  therefore $\ca' = L_{S_0}L_{S_2}(\ca) = \mc{T}\ca$ by Theorem \ref{double_tilts} (see figure \ref{fig:tilt_slicing}).

\begin{figure}[H]
    \centering
\begin{tikzpicture}
	\begin{pgfonlayer}{nodelayer}
		\node [style=none] (0) at (-4, 0) {};
		\node [style=none] (1) at (0, 0) {};
		\node [style=none] (2) at (4.25, 0) {};
		\node [style=none] (3) at (-3, 1.5) {};
		\node [style=none] (4) at (3, 0.75) {};
		\node [style=none] (7) at (4.5, 0.75) {$Z'(E)$};
		\node [style=none] (9) at (-4, 1.75) {$Z'(S_0[1])=Z'(S_2[1])$};
		\node [style=none] (17) at (3, 1.5) {};
		\node [style=none] (18) at (1.5, 3) {};
		\node [style=none] (21) at (3.75, 4.25) {};
		\node [style=none] (22) at (4, 1.75) {$Z'(\mc{T}(S_0))=Z'(\mc{T}(S_2))$};
		\node [style=none] (25) at (1.75, 3.5) {$Z'(\delta)$};
	\end{pgfonlayer}
	\begin{pgfonlayer}{edgelayer}
		\draw (0.center) to (1.center);
		\draw [shorten >=0pt, dashed] (1.center) to (2.center);
		\draw [style={<-}] (3.center) to (1.center);
		\draw [style={->}, in=195, out=15] (1.center) to (4.center);
		\draw [style=<-] (17.center) to (1.center);
		\draw [style=<-] (18.center) to (1.center);
	\end{pgfonlayer}
\end{tikzpicture}
    \caption{Central charges of simple objects and $E$, $\delta$ for $\sigma'$}
    \label{fig:tilt_slicing}
\end{figure}
Since $\ca'$ is of finite-length, and the phase of $E$ for $\sigma'$ is $\phi(S_0)+\epsilon-\phi(S_0)-\epsilon' = \epsilon-\epsilon'\in \Bigl(0,\phi\bigl(\mc{T}(S_0)\bigr)\Bigr)$, therefore $E$ is not semistable for $\sigma'$  by Lemma \ref{lem:cone}. Since $\sigma'$ and $\sigma$ contain the same set of semistable objects, so $E$ is not semistable for $\sigma$ either.

Similarly, suppose $E\in \mc{P}\left(\phi(S_1)-\epsilon\right)$, where $\epsilon \in \left(0, \phi(S_1)-\phi(\mc{T}^{-1}(S_1)) \right)$, we choose $0 < \epsilon' \ll \epsilon$. Let $t = 1+\epsilon'- \phi(S_1)$ and $\tau := \sigma\cdot t = (Z', \mc{P}')$. For $\tau$ we have the phases of 
\begin{eqnarray*}
    \phi'(S_1[-1]) = \phi'(S_3[-1]) &=& \epsilon' \in (0,1/2),\\
    \phi'\left(\mc{T}^{-1}(S_1)\right) = \phi'\left(\mc{T}^{-1}(S_3)\right) &=& \phi\left(\mc{T}^{-1}(S_1)\right) -\phi(S_1) +1+\epsilon' \in (1/2,1).
\end{eqnarray*}
The heart $\ca' = \mc{P}'(0,1]$ contains the  simple objects $\{S_1[-1],S_3[-1],\mc{T}^{-1}(S_1),\mc{T}^{-1}(S_3)\}$ which generate  $R_{S_1}R_{S_3}\ca$, $R_{S_1}R_{S_3}\ca\subset \ca'$ therefore $\ca' = R_{S_1}R_{S_3}\ca = \mc{T}^{-1}\ca$. The phase of $E$ for $\tau$ is $\phi(S_1)-\epsilon +1+\epsilon'-\phi(S_1) = 1+\epsilon'-\epsilon \in (\mc{T}^{-1}(S_1),1]$, therefore $E$ cannot be semistable for $\tau$ again by Lemma \ref{lem:cone}, and  is also not semistable for $\sigma$.

For the second statement, let $E\in \mc{P}\left(\phi(S_0)\right)$ be a stable object, we take the Jordan-H{\"o}lder filtration of $E$:
\[
 0\subset E_n\subset E_{n-1}\subset\cdots  E_1\subset E_0 = E,
\]
such that $ E_i/E_{i-1} = S_j $ for $j \in\{ 0,\cdots 3\}$. Since $\phi(S_0) = \phi(S_2) \neq \phi(S_1) = \phi(S_3)$, by the linearity of $Z$  the only graded factors appear in the filtrations are $S_0$ and $S_2$. Therefore $S_0$ or $S_2$ is a subobject of $E$, thus must be isomorphic to $E$. Similarly if $E\in \mc{P}\left(\phi(S_1)\right)$ is stable, we prove that $E$ is one of $S_1$ and $S_3$ exactly in the same way. Now we apply $\mc{T}^{\pm m}$ on $\sigma$ for $m\geq 0$.  Using Lemma \ref{g_act} again, we see $\mc{T}^{\pm m}(\sigma)$ and $\sigma$ contain the same set of stable objects. Therefore the stable objects of phase $\phi\bigl(\mc{T}^m (S_0)\bigr)$ are $\mc{T}^m(S_0)$ and $\mc{T}^m(S_2)$, and the stable objects of phase $\phi\bigl(\mc{T}^{-m} (S_1)\bigr)$ are $\mc{T}^{-m}(S_1)$ and $\mc{T}^{-m}(S_3)$. By Lemma \ref{lem:kron_type} they are of special Kronecker type I. 
\end{proof}
Let $\sigma \in \mc{U}^n(\mc{A})^{\Phi}_-$, we take $\Psi(\sigma) \in \mc{U}^n(\mc{A})^{\Phi}_+$, then by the above Lemma  the stable objects for $\sigma$ outside the ray $\phi= \frac{1}{2}$ are of special Kronecker type II.

The rest of this section is devoted to characterizing the stable objects on the ray $\phi = \frac{1}{2}$.
\begin{lem}\label{supp_points}
Let $\sigma \in\mc{U}^n(\mc{A})^{\Phi}_+$ and $E\in \mc{P}(\frac{1}{2})$, then $\pi_*E$ is (set theoretically) supported on $S \times \mb{P}^1$ where $S$ is a finite set of closed points in $\mb{P}^1$.
\end{lem}
\begin{proof}
According to Corollary \ref{inv_t}, $\mc{T}^n \left(\mc{P}(\frac{1}{2})\right) = \mc{P}(\frac{1}{2}) \subset \ca$ for   any integer $n$. Therefore  we have the vanishing of cohomology groups $\mr{Ext}^k (\mc{Q}, E) = 0$ for $k \neq 0$, in particular:
\begin{eqnarray*}
    0&=& \Ext^k_X\left(\pi^*\co \oplus \pi^*\co(1,1), E \otimes \pi^*\co(n,0)\right) \\
    &=&\Ext^k_Z\left(\co\oplus \co(1,1), \pi_* E\otimes \co(n,0)\right) \quad (\text{projection formula})\\
    &=&\mb{H}^k\bigl(\mb{P}^1, (p_1 \circ \pi)_*E\otimes \co(n)\bigr) \oplus \mb{H}^k\Bigl(\mb{P}^1, (p_1\circ \pi)_* \bigl(E(0,-1)\bigr)\otimes\co(n-1)\Bigr)
\end{eqnarray*}
where $\mb{H}$ means the hypercohomology of complexes and $k\neq 0$. In general for a complex $F^{\star}\in D^b(X)$, we have a spectral sequence  \cite[p.74]{Huy06}
$$
    E_2^{p,q} = \ho^q\left(X, H^p(F^{\star})\right) \Rightarrow \mb{H}^{p+q}(X, F^{\star}).
$$
We write $E'_0 = (p_1 \circ \pi)_* E$ and $E'_1 = (p_1 \circ \pi)_*\bigl(E(0,-1)\bigr)$. By taking $n\gg0$, then $\ho^i\left(\mb{P}^1, H^j(E'_m) \otimes \co(n)\right) = 0$ for $i \neq 0$, therefore the spectral sequence  degenerates, we have 
\begin{equation} \label{vanish_1}
    \mb{H}^k(\mb{P}^1, E'_m\otimes \co(n)) = \bigoplus_j \ho^{k-j}\left(\mb{P}^1,  H^j(E'_m)\otimes \co(n) \right) = 0  \quad \text{for } k\neq 0.
\end{equation}
Fix $j\neq 0$. Then $\ho^0\left(\mb{P}^1, H^j(E'_m)\otimes \co(n)\right) = 0$ where $n\gg 0$. This implies that $H^j(E'_m) = 0$. Therefore $E'_m$ is concentrated in degree $0$ and is indeed a sheaf.

For $m=0, 1$, now we have 
\begin{equation}
    \ho^k(\mb{P}^1, E'_m \otimes \co(n)) = 0, \quad k = 1,\quad n\in \mb{Z}.
\end{equation}
By taking $n\ll 0$, and using the fact that every coherent sheaf on $\mb{P}^1$ splits into line bundles and torsion sheaves\ \cite{CB}, we have $\mr{dim}(\mr{supp}\ E'_m ) = 0$. 

We denote  $S := \supp E'_0 \cup \supp E'_1$. Suppose $s\notin S$, we consider the following fibre product diagram with naturally-defined morphisms:
\begin{equation}
    \xymatrixcolsep{4pc}
    \xymatrix{
        p_1^{-1}(s) \ar@{^{(}->}[r]^i \ar[d]_{pt} & \mb{P}^1\times \mb{P}^1 \ar[d]^{p_1} \\
        \{s\} \ar@{^{(}->}[r]_j & \mb{P}^1
    }
\end{equation}
We apply the flat base change theorem \cite[Chapter 3.3]{Huy06} to $\pi_*E, \ \pi_*\bigl(E(0,-1)\bigr)\in D^b(Z)$, for any integer $k$
\begin{eqnarray*}
     m = 0:& \quad &\ho^k(\mb{P}^1, i^*\pi_* E) = j^*E'_0 = 0 \\
     m = 1:& \quad & \ho^k\left(\mb{P}^1, i^*\pi_*(E(0,-1))\right) = \ho^k\left(\mb{P}^1,(i^*\pi_* E)\otimes \co(-1)\right) = j^* E'_1 = 0.
\end{eqnarray*}
These vanishings imply that $i^*\pi_* E = 0$ (one can again use the structure theorem of coherent sheaf on $\mb{P}^1$). Therefore $\pi_*E$ is supported on $S \times \mb{P}^1$.
\end{proof}

\begin{lem}\label{lem:stab_pushfor}
    Suppose $E\in \ca$ is isomorphic to the shift of a sheaf, and $\End_{\ca}(E) \cong \mb{C}$ then it is the pushforward $E = s_* F[i]$ for some $F\in \Coh \ff$.
\end{lem}
\begin{proof}
     We follow the idea in the proof of \cite[Lemma 3.1]{BaM11}: let $Y$ be the scheme-theoretic support of $E$. By definition, $\ho^0(\co_Y)$ acts faithfully on $E$, and $\End_{\ca}(E)=\mb{C}. \mr{Id}$, therefore $\ho^0(\co_Y)\cong \mb{C}$. Take the composition of the embedding of $Y$  with the contraction
    \[
       f:Y\hookrightarrow  X\twoheadrightarrow \bar{X} = \Spec \ho^0(\co_X), 
    \]
    as $ \ho^0(f_*\co_Y) = \ho^0(\co_Y) =\mb{C}$, so the scheme-theoretic image of $Y$ under $f$ is a point of $\bar{X}$. By definition of $f$, the point will be the origin (singular point), thus $Y$ is contained scheme-theoretically in the fiber of the contraction map. Since the scheme-theoretic fiber of the origin is exactly $\ff$, so $E = s_*F$ for some $F\in \Coh \ff$.
\end{proof}

Now we are able to characterize the stable objects on the ray $\phi = \frac{1}{2}$.
\begin{thm} \label{stabobj}

    {\em If $\sigma \in\mc{U}^n(\mc{A})^{\Phi}_+$}, and $E$ is a  $\sigma$-stable object in $\mc{P}(\frac{1}{2})$, then there exists a point $x\in \mb{P}^1$ such that either $E = F_1(x) = s_* \co_{\{x\}\times \mb{P}^1}$ or $E = F_2(x) = s_* \co_{\{x\}\times \mb{P}^1}(-1)[1]$.
    
    {\em If $\tau\in \mc{U}^n(\mc{A})^{\Phi}_-$},  and $E$  is a $\tau$-stable object in $\mc{P}(\frac{1}{2})$, then there exists a point $x\in \mb{P}^1$ such that either $E = \Psi\left(F_1(x) \right)$ or $E = \Psi\left(F_2(x)\right)$.
\end{thm}
\begin{proof}
We prove the first statement, the second statement follows since $\Psi$ exchanges the stability conditions in $\mc{U}^n(\mc{A})^{\Phi}_+$ and $\mc{U}^n(\mc{A})^{\Phi}_-$, as also exchanges the objects of special Kronecker type I and II.

Suppose $E\in \mc{P}\left(\frac{1}{2}\right)$ is stable and not isomorphic to $F_1(x_1)$ and $F_2(x_1)$ for any $x_1\in\mb{P}^1$. We will show that there is a vector bundle $F\in \Coh \ff$ such that $E\cong s_*F[1]$. We follow the idea of proof in \cite[Lemma 3.2]{BaM11}: since $F_i$ are stable, therefore
\[
    \Hom_X(E, F_n(x_1)) = \Hom_X(F_n(x_1), E) = 0,\ n=1,\ 2.
\]
Note that we have the short exact sequence in $\mc{A}$:
\[
    0\rightarrow F_1(x_1) \rightarrow \co_x \rightarrow F_2(x_1) \rightarrow 0,
\]
where $x\in \ff$ such that $p_1(x) = x_1$, therefore there cannot be any nonzero map $E \rightarrow \co_x$ or $\co_x\rightarrow E$. Since $\co_x$ is semistable of phase $1/2$, then $\Hom^i_X(E, \co_x) = 0$ for $i \leq 0$, and Serre duality gives $\Hom_X(\co_x[i],E) = \Hom_X(E,\co_x[i+3])= 0$ for $i\geq 0$ and $x\in \ff$. Since $E$ is supported on $\ff$, there will be no homomorphisms with shifts of skyscraper sheaves outside the zero-section. 
Therefore we can apply \cite[Proposition 5.4]{BM02} and deduce that $E$ is isomorphic to a two-term complex of locally-free sheaves
\[
    E^{-2}\xrightarrow{d^{-2}} E^{-1}. 
\]
Hence $H^{-2}(E)\subset E^{-2}$ is torsion free on $X$. However, since $H^{-2}(E)$ is supported on $\ff$, therefore it must vanish. The map $d^{-2}$ is injective, so that $E$ is isomorphic to the shift of a sheaf $F'[1]$. Since $E$ is stable, $\End(F')\cong \mb{C}.\mr{Id}$, therefore we apply 
 Lemma \ref{lem:stab_pushfor} and show that $F' = s_*F$ where $F\in \Coh \ff$. Since $\Hom_X(s_*\co_x, s_*F[1]) \cong \Hom_{\ff}(\co_x,F[1]) \oplus \Hom_{\ff}(\co_x, F) = 0$, therefore $F$ has depth $2$ and by Auslander-Buchsbaum formula, $F$ is actually locally free.

 However, this contradicts with Lemma \ref{supp_points} which says that $\pi_*E$ is supported on a $S\times \mb{P}^1\subset \pp$, where $S$ is a finite set of points. Therefore we conclude that $E$ is isomorphic to either $F_1(x)$ or $F_2(x)$ for some $x\in\mb{P}^1$.
\end{proof}
In summary, we have completed the description of the stable objects for $\sigma \in \mc{U}^n(\ca)^{\Phi}$:
\begin{thm} \label{thm:stab_obj}
	\begin{enumerate}
		\item If $\sigma \in \mc{U}^n(\mc{A})^{\Phi}_+$, then the stable objects (up to a shift) are of special Kronecker type I, and the classes of stable objects (up to a sign) in $K_0(\ca)$  are ($n\in \mb{N}$)
            \begin{eqnarray*}
                n \gamma_0+(n+1)\gamma_1,& (n+1)\gamma_0+n\gamma_1,\\
                n \gamma_2+(n+1)\gamma_3,& (n+1)\gamma_2+n\gamma_3,\\
                \gamma_0+\gamma_1,& \gamma_2+\gamma_3.
            \end{eqnarray*}
        \item If $\sigma \in \mc{U}^n(\mc{A})^{\Phi}_-$, then the stable objects (up to a shift) are of special Kronecker type II, and the classes of stable objects (up to a sign)  in $K_0(\ca)$ are ($n\in \mb{N}$)
        \begin{eqnarray*}
                n \gamma_1+(n+1)\gamma_2,& (n+1)\gamma_1+n\gamma_2,\\
                n \gamma_3+(n+1)\gamma_0,& (n+1)\gamma_3+n\gamma_0,\\
                \gamma_1+\gamma_2,& \gamma_3+\gamma_0.
            \end{eqnarray*}
	\item If $\sigma \in \mc{U}^n(\ca)^{\Phi}$ and  $\phi(S_i) = \frac{1}{2}$ for each $i$, then the stable objects (up to a shift) are only $\{S_i\}_{i}$, and the classes of stable objects (up to a sign) in $K_0(\ca)$ are
        \[
            \gamma_0, \ \gamma_1,\ \gamma_2,\ \gamma_3.
        \] 
	\end{enumerate}
\end{thm}
\begin{proof}
    We have proved the first two cases in the above. For the last case, we  do induction on the length $l(E)$ of object $E\in \ca$. When $l(E) = 1$, it is obvious. Then for $l(E)= n+1$, by taking the Jordan-H{\"o}lder filtration of $E$, we have short exact sequence
    \[
        0\rightarrow E' \rightarrow E \rightarrow S_i^{\oplus n_i}\rightarrow 0.
    \] 
    Then $E'$ is semistable by our induction hypothesis, note that $E'$ has the same phase as $S_i$,  we see that $E$ is also semistable of phase of $S_i$. Therefore we proved that there are no other stable objects other that $S_i$, $i=0,\cdots, 3$.
\end{proof}
By applying $\mb{C}$-action we obtain the same description of stable objects for general stability conditions in $\mc{U}(\ca)^{\Phi}$.

\section{Space of invariant stability conditions} \label{sec:spe_stab}
Recall from the introduction  the subset $\Delta\subset \overline{K_0(\ca)}$ is defined to be the set of classes of stable objects for $\sigma\in \mc{U}(\ca)^{\Phi}$ in the quotient group $\overline{K_0(\ca)} = K_0(\ca)/K_0(\ca)^{-\varphi}$.  By Theorem \ref{thm:stab_obj} $\Delta$ consists of the following elements:
\[
    \Delta = \bigl\{n\in \mb{N}: n\gamma_0+(n+1)\gamma_1,\ (n+1)\gamma_0+n\gamma_1, \pm(\gamma_0+\gamma_1) \bigr\}. 
\]
Recall also that 
    $$
        \mc{H}^{\reg} := \Hom(\overline{K_0(\ca)}, \mb{C}) \setminus \bigcup_{\pmb{v}\in \Delta} \pmb{v}^{\perp},
    $$
where $\pmb{v}^{\perp} := \{Z\in \Hom(\overline{K_0(\ca)}, \mb{C}) | Z(\pmb{v}) = 0 \}$ is the hyperplane complement.

Note that $\mc{H}^{\reg}$ is the complement of a  family of hyperplanes in $\Hom(\overline{K_0(\ca)}, \mb{C})\cong \mb{C}^2$:
\begin{eqnarray*}
        n Z(\gamma_0) + (n+1)Z(\gamma_1) &=& 0 \\
        (n+1) Z(\gamma_0) + n Z(\gamma_1) &=& 0 \\
        Z(\gamma_0)+Z(\gamma_1) &=& 0
\end{eqnarray*}
for $n \geq 0$ (see Figure \ref{fig:hreg}).

In the final section we prove Theorem \ref{thm:main_thm2} in the introduction: the forgetful map 
\[
    \mc{Z}: \bigl(\Stab(X)^{\Phi}\bigr)_0 \rightarrow \Hom(\overline{K_0(X)},\mb{C})
\]
factors through
\[
\mc{Z}: \bigl(\Stab(X)^{\Phi}\bigr)_0 \rightarrow \mc{H}^{\reg}.
\]
Moreover, the above is a covering map. 

Recall that a continuous map $f:A\rightarrow B$ between topological spaces is called a covering map, if every point $b\in B$ has an open neighborhood $V\subset B$ such that the restriction of $f$ to each connected component of $f^{-1}(V)$ is a homeomorphism onto $V$.

We first analyze the boundary of $\mc{U}^n(\ca)^{\Phi}$. Recall $\{S_i\}_{i\in \mb{Z}_4}$ (Corollary \ref{simp_heart}) are the simple objects which generate $\ca$ and $\gamma_i = [S_i]$ are their classes. By definition $\partial\ \mc{U}^n(\ca)^{\Phi}$ has four components of codimension-one submanifolds (real lines), which are 
\[
W^+_i := \{Z(\gamma_i) = Z(\gamma_{i+2})\in \mb{R}_{>0}\},\quad  W^-_i:=\{Z(\gamma_i)=Z(\gamma_{i+2}) \in \mb{R}_{<0}\}
\]
$i=0,\ 1$. 
Though we cannot apply Lemma \ref{lem:stab_tilt} directly, however, since we are deforming $\sigma$ while preserving the condition $Z(\gamma_i) = Z(\gamma_{i+2})$,  the statement and proof are exactly the same as there.
\begin{lem}\label{lem:boundary}
\begin{enumerate}
    \item For any stability condition on $W^+_i$ ($i=0,\ 1$) there exists an open neighborhood $V$ such that $V\subset \UnA{}\cup \UnA{L_{S_i}L_{S_{i+2}}}$. Similarly, for any stability condition on $W^-_i$ ($i=0,\ 1$) there exists an open neighborhood $V$ such that $V\subset \UnA{}\cup \UnA{R_{S_i}R_{S_{i+2}}}$.
    \item We have $W^+_i = \overline{\mc{U}^n}(\ca)^{\Phi}\cap \overline{\mc{U}^n}(\tau\ca)^{\Phi}$, where $\tau = \mc{T}$ when $i =0$ and $\tau = \mc{T}_{\Psi}$ when $i=1$. Similarly, $W^-_i = \overline{\mc{U}^n}(\ca)^{\Phi}\cap \overline{\mc{U}^n}(\tau\ca)^{\Phi}$, where $\tau = \mc{T}_{\Psi}^{-1}$ when $i=0$ and $\tau = \mc{T}^{-1}$ when $i=1$.
\end{enumerate}
\end{lem}
\begin{proof}
    In the following proof we will repeatedly use Lemma \ref{nest_t} that if $\ca,\ \ca'\subset \CD$ are hearts of bounded t-structures and $\ca\subset \ca'$, then $\ca = \ca'$. 

First we suppose $\sigma \in W^+_0$, that is $Z(\gamma_0) = Z(\gamma_2)\in \mb{R}_{>0}$. 
The objects $\wtil{S_1} = \mc{T}(S_0)$ and $\wtil{S_3} = \mc{T}(S_2)$ lie in $\ca$, and are in the short exact sequences by the computations in Section \ref{sec:double_tilts}:
\begin{equation}
    0\rightarrow S_{i-1}^{\oplus 2} \rightarrow \wtil{S_i} \rightarrow S_{i} \rightarrow 0,\quad i=1,\ 3
\end{equation}
where $2 = \dim_{\mb{C}}\Ext^1(S_i,S_{i-1})^*$. Since $\Hom(\wtil{S_1},S_0) = \Hom(\wtil{S_3},S_2) = 0$ the objects $\wtil{S_i}$ lie in $\mc{P}(0,1)$, and by choosing a small enough open neighborhood $V$ of $\sigma$ we can assume this is the case for all stability conditions $(Z,\mc{P})$ of $V$. We can split $V$ into two pieces 
\[
V_+=\{ \mr{Im} Z(S_0) = \mr{Im} Z(S_2) > 0\},\quad  V_- = \{\mr{Im}\ Z(S_0) = \mr{Im} Z(S_2) \leq 0\}.
\]
For $\sigma \in V_+$, we can shrink $V$ if necessarily such that $S_i\in \mc{P}(0,1)$ for all $i$. This shows that $\ca \subset \mc{P}(0,1]$ for all stability conditions in $V_+$, therefore $\mc{P}(0,1]=\ca$ and so $V_+\subset \mc{U}^n(\ca)^{\Phi}$. On the other hand,  for any stability condition $(Z,\mc{P})\in V_-$ the objects $S_0$ and $S_2$ are in $\mc{P}(-1/2,0]$, thus the heart $\mc{P}(0,1]$ contains the objects $S_0[1],\ S_2[1],\ \wtil{S_1}$
 and $\wtil{S_3}$. Since these are the simple objects of the finite length category $L_{S_0}L_{S_2}\ca$, therefore $\mc{P}(0,1] = L_{S_0}L_{S_2}\ca$ and so $V_-\subset \mc{U}^n(L_{S_0}L_{S_2}\ca)^{\Phi}$. Therefore $V\subset \mc{U}^n(\ca)^{\Phi}\cup \mc{U}^n(L_{S_0}L_{S_2}\ca)^{\Phi}$. 
 
 By applying $\Psi$ on $\sigma$ then $\Psi(\sigma)\in W^+_1$, that is $Z(\y_1)=Z(\y_3)\in\mb{R}_{>0}$, then there exists an open neighborhood $V$ of $\Psi(\sigma)$ such that $V\subset \UnA{}\cup \UnA{L_{S_1}L_{S_3}}$.

The proof for $\sigma \in W^-_i$ is essentially the same by replacing the left double tilt with the right double tilt. 

For the second statement, by the first part we have 
\[
    W^+_i = \overline{\mc{U}^n}(\ca)^{\Phi}\cap \overline{\mc{U}^n}(L_{S_i}L_{S_{i+2}}\ca)^{\Phi},\quad W^-_i = \overline{\mc{U}^n}(\ca)^{\Phi}\cap \overline{\mc{U}^n}(R_{S_i}R_{S_{i+2}}\ca)^{\Phi}.
\]
By Theorem \ref{double_tilts} we obtain the results.
\end{proof}

We denote by  $H$ the subgroup in $\Aut_*(D^b_0(X))$ (see Definition \ref{def:aut_group}) generated by $\mc{T}$ and $\mc{T}_{\Psi}$. 
\begin{prop} \label{lem:fund_domain_norm}
\begin{equation} \label{stab_norm}
    \bigcup_{g\in H}  \overline{\mc{U}^n}(g\ca)^{\Phi} = \bigl(\Stab(X)^{\Phi}\bigr)_n.
\end{equation}
\end{prop}
\begin{proof}
    
Since $t(\delta) = t_{\psi}(\delta) = \delta$ due to Lemma \ref{auto_kgroup}, therefore $\overline{\mc{U}^n}(g\ca)^{\Phi} \subset \bigl(\Stab(D^b_0(X))^{\Phi}\bigr)_n$ for any $g\in H$. 

Now we show that the left side is open and closed, hence the inclusion is in fact an equality. 

First we prove the {\em openess}. For any $\sigma\in \overline{\mc{U}^n}(g\ca)^{\Phi}$, considering the preimage of $\sigma$ under the autoequivalence $g$,  $\sigma = (Z,\mc{P})$ lies in $\overline{\mc{U}^n}(\ca)^{\Phi}$. Suppose first that $\mr{Im} Z(\gamma_i) > 0$ for each $i$, then we can choose an open neighborhood $U$ of $\sigma$ such that each simple object $S_i$ has phase $(0,1)$ for all stability conditions $(Z,\mc{P})$ of $U$. Since $\ca$ is the smallest extension-closed subcategory of $\CD$ containing $S_i$ it follows that $\ca \subset \mc{P}(0,1]$ of all stability conditions in $U$. Therefore $\mc{P}(0,1] = \ca$  by Lemma \ref{nest_t} and so $U$ is contained in $\mc{U}^n(\ca)^{\Phi}$.

Now suppose $\sigma$ lies on the boundary of $\mc{U}^n(\ca)^{\Phi}$, according to Lemma \ref{lem:boundary}, there is an open neighborhood $V$ of $\sigma$ such that $V\subset \UnA{}\cup \UnA{\tau }$ where $\tau$ is one of the autoequivalences $\mc{T}^{\pm 1}$ and $\mc{T}_{\Psi}^{\pm 1}$. This finishes the proof of openess.

To check the left side of (\ref{stab_norm}) is closed,  we only need to show the collection of closed sets is locally finite. Suppose 
\[
    \sigma \in \bigcap_{g\in H' \subset H} \overline{\mc{U}^n}(g\ca)^{\Phi}.
\]
It is obvious that $\mc{U}^n(g\ca)^{\Phi} \cap \mc{U}^n(g'\ca)^{\Phi} = \emptyset$ if $g\ca\neq g'\ca$. Taking the preimage under some autoequivalence $g$, we suppose $\sigma$ lies on the boundary of $\mc{U}^n(\ca)^{\Phi}$. Then this intersection is finite since by Lemma \ref{lem:boundary}, each boundary component corresponds to exactly one of the autoequivalences $\mc{T}^{\pm 1}$ and $\mc{T}_{\Psi}^{\pm 1}$.

This finishes the proof of the equality (\ref{stab_norm}).
\end{proof}

\begin{lem} \label{lem:fund_domain}
For any  stability condition $\sigma = (Z, \mc{P}) \in \bigl(\Stab(X)^{\Phi}\bigr)_0$, we have $Z(\delta) \neq 0 $ where $\delta = [\co_x]$ for $x\in Z$. Moreover, there is an exact equivalence $W\in H$ and $\lambda \in \mb{C}$ such that $\lambda W(\sigma)\in \overline{\mc{U}}(\ca)^{\Phi}$, the closure of $\mc{U}(\ca)^{\Phi}$.
\end{lem}
\begin{proof}
Suppose there exists $\sigma = (Z_1, \mc{P}_1) \in \bigl(\Stab(X)^{\Phi}\bigr)_0$ such that $Z_1(\delta) = 0$. We take an open neighborhood $U_{\sigma}$ of $\sigma$ and let $\tau = (Z_2,\mc{P}_2)\in U_{\sigma}$ be any stability condition such that $Z_2(\delta)\neq 0$. We can normalize $\tau$ by some $\lambda\in\BC$. Then by (\ref{stab_norm}) there exists some $g\in H$ such that $\lambda.g(\tau)\in \overline{\mc{U}^n}(\ca)^{\Phi}$. Now we choose $\co_x$ for some $x\in \mb{F}_0$, then $g(\co_x)$ is semistable for stability condition $\tau$. We have $\bigl[g(\co_x)\bigr] = \delta$ since $t$ and $t_{\psi}$ preserve $\delta$. Note that $g(\co_x)$ cannot be semistable for the stability condition $\sigma$.  However, this contradicts Lemma \ref{stab_sstable} since semistability is a closed condition.

Therefore for any stability condition $\sigma = (Z, \mc{P})$, we have $Z(\delta) \neq 0$. Then by choosing $\lambda \in \mb{C}$ such that $\lambda\cdot Z(\delta) = i$, and by (\ref{stab_norm}) we can find $W\in H$ such that $\lambda W(\sigma) \in \overline{\mc{U}^n}(\ca)^{\Phi}\subset \overline{\mc{U}}(\ca)^{\Phi}$. This finishes the proof.
\end{proof}

The following lemma is an easy consequence of Lemma \ref{auto_kgroup}: 
\begin{lem}\label{lem:reg_pre}
The automorphisms $t^{\pm 1}$ and $t_{\psi}^{\pm 1}$ preserve $\Delta$.
\end{lem}
\begin{cor}\label{cor:delta}
    Let $E$ be a stable object for $\sigma \in \bigl(\Stab(X)^{\Phi}\bigr)_0$, then the class  $[E] \in \overline{K_0(\ca)}$ lies in $\Delta$. 
\end{cor}
\begin{proof}
    By the above lemma, there is an autoequivalence $W\in H$ and $\lambda\in\mb{C}$ such that $\lambda W(\sigma)$ lies in the closure of $\mc{U}(\ca)^{\Phi}$. Since the stable objects remain stable in an open neighborhood $V$ of $\lambda W(\sigma)$, we choose $\sigma'\in V$ such that $\sigma'\in\mc{U}(\ca)^{\Phi}$, it follows that $\sigma'$ and $\lambda W(\sigma)$ contain the same set of stable objects.  Therefore the classes of stable objects for $\lambda W(\sigma)$ lie in $\Delta$ by Theorem \ref{thm:stab_obj}. Since by Lemma \ref{lem:reg_pre} the group element in $H$ preserves $\Delta$, therefore $[E] \in \Delta$.
\end{proof}
\begin{thm}\label{thm:local_homeo}
The image of the local homeomorphism 
\[
\mc{Z}: \bigl(\Stab(X)^{\Phi}\bigr)_0 \rightarrow \Hom(\overline{K_0(\ca)}, \mb{C})
\]
lies in $\mc{H}^{\reg}$.
\end{thm}
\begin{proof}
By Corollary \ref{cor:delta}, the set of class of any stable object $E$ for stability condition $\sigma = (Z,\mc{P})$ is exactly $\Delta$. Since $Z(E) \neq 0$, therefore $\mc{Z}(\sigma) \in \mc{H}^{\reg}$.
\end{proof}
We fix a norm $\lVert\ \cdot\  \rVert$ on $\overline{K_0(\mc{A})}_{\mb{R}} = \overline{K_0(\ca)}\otimes_{\mb{Z}} \mb{R}$. The induced norm on $\Hom(\overline{K_0(\ca)},\mb{C})$ is denoted by $\lVert\ \cdot\ \rVert^{\vee}$.
\begin{lem}\label{lem:supp_prop}
Let  $Z\in \mc{H}^{\reg}$, there exists a constant $C > 0$ (depending on $Z$) such that 
\begin{equation} \label{supp_property}
    \lVert \pmb{v} \rVert \leq C |Z(\pmb{v})|
\end{equation}
for all $\pmb{v}\in \Delta$.
\end{lem}
\begin{proof}
Since all norms over finite dimensional space are equivalent, we might take 
$$
    \lVert \pmb{v} \rVert^2 = v_0^2 + v_1^2,
$$
where $v_i$ denotes the $i$-th component of a vector $\pmb{v} \in \overline{K_0(\mc{A})}_{\mb{R}}$ with respect to the basis $\gamma_0,\ \gamma_1$. For any vector $\pmb{v}$ with the class in $\Delta$, we have
\[
    \lVert \pmb{v}\rVert^2 = n^2+(n+1)^2 \text{ or } 2,
\]
for $n\geq 0$.
Suppose $\lVert \pmb{v}\rVert^2 = 2$, in this case  $\pmb{v} = \gamma_0 + \gamma_1$. Since $Z(\gamma_0 + \gamma_1)\neq 0$ by definition. Therefore we can choose a constant $c_0$ such that (\ref{supp_property}) holds. 

Suppose $\arg Z(\gamma_0) \neq \arg Z(\gamma_1)$. Since $Z(\gamma_i) \neq 0$, without loss of generality we take  a suitable $\GL(2,\mb{R})$-action on the complex plane, such that $Z(\gamma_0) = 1,\ Z(\gamma_1) = i$. Now suppose $\lVert \pmb{v}\rVert^2 = n^2+(n+1)^2$, then in this case $\pmb{v} = n\gamma_0 + (n+1)\gamma_1$ or $(n+1)\gamma_0+n\gamma_{1}$ in $\Delta$, therefore $|Z(\pmb{v})|^2 = n^2 +(n+1)^2$. So we have
\[
    \lVert \pmb{v} \rVert = |Z(\pmb{v})|.
\]
We can take $c_1 \geq 1$.

Suppose $\arg Z(\gamma_0) = \arg Z(\gamma_1)$. Then there exists a constant $c_2$ such that 
\[
	 \lVert \pmb{v} \rVert = c_2 |Z(\pmb{v})|.
\]
Finally we choose the maximum from $c_i$ such that (\ref{supp_property}) holds for any $\pmb{v}\in\Delta$.
\end{proof}

\begin{thm}[Covering property] \label{thm:stab_cover}
$\mc{Z}: \bigl(\Stab(X)^{\Phi}\bigr)_0\rightarrow \mc{H}^{\reg}$ is a covering map.
\end{thm}
\begin{proof}
We first show that $\mc{H}^{\reg}$ is open.  Let $Z \in \mc{H}^{\reg}$. Lemma \ref{lem:supp_prop} shows that there is a constant $C > 0$ (depending on $Z$) such that 
\[
    \lVert \pmb{v} \rVert \leq C |Z(\pmb{v})|
\]
for all $\pmb{v} \in \Delta$. Given $\epsilon >0$, we define an open subset 
\[
    B_{\epsilon}(Z) = \left\{W\in \Hom(\overline{K_0(\ca)},\mb{C}):\lVert W-Z\rVert^{\vee} < \epsilon/C \right\} \subset \Hom(\overline{K_0(\ca)},\mb{C}).
\]
Then for $W\in B_{\epsilon}(Z)$, we have 
\[
|W(\pmb{v})-Z(\pmb{v})|\leq \lVert W-Z\rVert^{\vee}\lVert\pmb{v}\rVert <  \epsilon|Z(\pmb{v})|
\]
for $\pmb{v}\in\Delta$. Therefore if $\epsilon < 1$ then any $W\in B_{\epsilon}(Z)$ satisfies $W(\pmb{v}) \neq 0$ for $\pmb{v} \in \Delta$. Hence $W\in \mc{H}^{\reg}$, this shows that $\mc{H}^{\reg}$ is open. 

Now we fix a positive real number $ \epsilon_0 < \frac{1}{8} $ and assume that $\epsilon < \sin(\pi\epsilon_0)$. Given any $\sigma = (\mc{Z},\mc{P})\in\bigl(\Stab(X)^{\Phi}\bigr)_0$ with $\mc{Z}(\sigma) = Z$, we define the open neighborhood of $\sigma$
$$
    C_{\epsilon}(\sigma) = \left\{ \tau = (W,\mc{Q})\in \mc{Z}^{-1}(B_{\epsilon}(Z)): d(\mc{P}, \mc{Q}) < 1/2 \right\}, 
$$
where $d(-,-)$ is defined in Definition \ref{defi:slice_metric}. By Lemma \ref{stab_equi}, the map
\begin{equation}\label{eqn:map1}
    \mc{Z}:C_{\epsilon}(\sigma) \rightarrow B_{\epsilon}(Z)  
\end{equation}
is injective. Let $W\in B_{\epsilon}(Z)$, then for any $E$ stable for $\sigma$, by Corollary \ref{cor:delta},  we have 
\begin{eqnarray*}
    |W(E) - Z(E)| < \sin(\pi \epsilon_0) |Z(E)|.
\end{eqnarray*}
Using the deformation result Theorem \ref{stab_deform}, we conclude that there is a unique stabilty condition $\tau = (W,\mc{P}') \in C_{\epsilon}(\sigma) $ such that $\mc{Z}(\tau) = W$ and $d(\mc{P},\mc{P}') <\epsilon$. Thus the map (\ref{eqn:map1}) is a homeomorphism. For each $\sigma\in\mc{Z}^{-1}(Z)$, we prove $C_{\epsilon}(\sigma)$ is mapped homeomorphically by $\mc{Z}$ onto $B_{\epsilon}(Z)$ exactly in the same way. 

Finally we check that 
\begin{equation}
    \mc{Z}^{-1}\left(B_{\epsilon}(Z)\right) = \bigcup_{\sigma\in \mc{Z}^{-1}(Z)} C_{\epsilon}(\sigma) 
\end{equation}
is disjoint. Suppose there exists $\tau = (W, \mc{Q}) \in C_{\epsilon}(\sigma) \bigcap C_{\epsilon}(\sigma')$, where we denote by $\sigma = (Z, \mc{P})$ and $\sigma' = (Z,\mc{P}')$. Then 
\begin{equation*}
    d(\mc{P}, \mc{P}') \leq d(\mc{P}, \mc{Q}) + d(\mc{Q}, \mc{P'})  <  1.
\end{equation*}
Therefore by Lemma \ref{stab_equi} again, we have $\sigma = \sigma'$, which means $C_{\epsilon}(\sigma) = C_{\epsilon}(\sigma')$. We have finished the proof.
\end{proof}

{\bf Acknowledgements.}
The work forms part of my PhD thesis.  I am very grateful to my supervisor Tom Bridgeland for sharing this question generously with me and helping translate the physicists' language into math. I would also like to thank Fabrizio Del Monte and Ananyo Dan for helpful discussions. Part of the work was finished while I was visiting Tianyuan Mathematical Center in Southwest China, Sichuan University. I appreciate Prof. Xiaojun Chen for his invitation and hospitality. The work was partially supported by China scholarship council-University of Sheffield, and Sichuan Science and Technology Program (No. 2025ZNSFSC0800).

\bibliographystyle{plain}

\end{document}